\documentclass[journal,final,onecolumn,12pt]{IEEEtran}
\usepackage{amsthm} % for self-defined new theorems
\usepackage{amsmath} % for complicated math formulas
\usepackage{amssymb}  % assumes amsmath package installed
\usepackage{enumerate}
\usepackage{enumitem}

\usepackage{algorithm}
\usepackage{algpseudocode}

\usepackage{tabularx}

\usepackage[pdftex]{graphicx}
\usepackage[outdir=./figures/]{epstopdf}
\graphicspath{{./figures/}}

\usepackage{caption}
\usepackage{subcaption}

\usepackage{xcolor}
\usepackage{flushend}

\usepackage{tablefootnote}

\newtheorem{definition}{Definition}
\newtheorem{proposition}{Proposition}
\newtheorem{problem}{Problem}
\newtheorem{assumption}{Assumption}
\newtheorem{remark}{Remark}
\newtheorem{theorem}{Theorem}

\newtheorem{example}{Example}

\newcommand{\set}[1]{\left\{#1\right\}}
\newcommand{\Real}{\mathbb R}
\newcommand{\Z}{\mathbb Z}

\newcommand{\M}{\mathcal M}
\newcommand{\T}{\mathcal T}
\newcommand{\Q}{\mathcal Q}
\newcommand{\A}{\mathcal A}
\renewcommand{\P}{\mathcal P}
\newcommand{\Y}{\mathcal Y}
\newcommand{\C}{\mathcal{C}}

\newcommand{\sv}{\,\vert\;}

\newcommand{\ra}{\rightarrow}

\newcommand{\I}{\mathcal{I}^{\infty}}

\newcommand{\pre}{\text{Pre}}

\begin{document}
\title{\LARGE\bf Invariance Control Synthesis for Switched Systems:\\ An Interval Analysis Approach
\thanks{Y. Li and J. Liu are with the Department of Applied Mathematics, University of Waterloo, Waterloo, Ontario, Canada. (e-mail: yinan.li@uwaterloo.ca, j.liu@uwaterloo.ca)}% <-this % stops a space
\thanks{Preliminary versions of this work \cite{li2016computing,li2016interval} are to be presented in the IEEE Multi-Conference on Systems and Control (MSC) 2016 and IEEE Conference on Decision and Control (CDC) 2016.}
}

\author{Yinan Li,~\IEEEmembership{Student Member,~IEEE,}
        Jun Liu,~\IEEEmembership{Member,~IEEE}}% <-this % stops a space

\maketitle

\begin{abstract}
This paper focuses on the invariance control problem for discrete-time switched nonlinear systems. The proposed approach computes controlled invariant sets in a finite number of iterations and directly yields a partition-based invariance controller using the information recorded during the computation. In contrast with Lyapunov-based control methods, this method does not require the subsystems to have common equilibrium points. Algorithms are developed for computing both outer and inner approximations of the maximal controlled invariant sets, which are represented as finite unions of intervals. The general convergence results of interval methods allow us to obtain arbitrarily precise approximations without any stability assumptions. In addition, invariant inner approximations can be computed provided that the switched system satisfies a robustly controlled invariance condition. Under the same condition, we also prove the existence of an invariance controller based on partitions of the state space. Our method is illustrated with three examples drawn from different applications and compared with existing work in the literature.
\end{abstract}


\IEEEpeerreviewmaketitle

\section{Introduction}
\IEEEPARstart{S}{witched} systems are dynamical systems consisting of multiple subsystems or modes. The mode to be activated at a specific time is determined by a discrete variable. A variety of real-world systems belong to this kind, e.g. electrical power converters \cite{FribourgBook13} and DISC engines \cite{RinehartDRK08}. Control synthesis for systems with complex dynamics or specifications, such as robot motion planning \cite{KressGazitFP09} and flight management \cite{FrazzoliDF05}, is usually simplified to switching control between different operating modes and motion primitives.

Invariance control concerns with the problem of finding a control law such that the solutions of a closed-loop system are restricted to a specified region in the state space for all future time. Practical stabilization and safety control, which are two important control specifications in various settings, are essentially in the scope of invariance control. In practical stabilization, system states have to be controlled and maintained in a sufficiently small region around a given set point. In safety control, the region to be rendered invariant through admissible control inputs is defined by some safe operating envelope. Therefore, invariance control plays an important role in multiple branches of control, such as constrained, robust, and optimal control \cite{Kerrigan00}.

In this paper, we propose an invariance control synthesis method for discrete-time switched nonlinear systems. Our approach does not assume that the subsystems are asymptotically stable or have common equilibrium points. The underlying idea is to compute controlled invariant sets using interval methods. More specifically, successive approximations of controlled invariant sets are computed and represented as unions of intervals. For any user-defined precision, such computation is guaranteed to terminate and converge to the true invariant sets as the chosen precision tends to zero. An invariance controller can be extracted once the computation of the controlled invariant set terminates.

\subsection{Related work}

\subsubsection{Computation of maximal controlled invariant sets}
A key objective of invariance control synthesis is to find the maximal controlled invariant set within a given region in the system state space. To determine if a set is controlled invariant for a discrete-time system, simply checking the system vector flow on the set boundary is insufficient \cite{Blanchini99}. As proposed in \cite{Bertsekas72}, the invariance control solution for general discrete-time dynamical systems relies on a fixed-point algorithm. Starting with a target set of states, this algorithm iteratively eliminates the states that violate the invariance condition until the maximal controlled invariant set. i.e., the fixed point, is achieved. This method has been applied to discrete-time linear \cite{Blanchini94,Kerrigan00} and switched linear systems \cite{BenSassiG12,Dehghan12,FribourgBook13}. Practical realization of the conceptual fixed-point algorithm, however, is still a challenging and open problem.

\emph{Linear systems:} Even for linear systems, computing exact maximal controlled invariant sets is nontrivial, primarily for the lack of finite termination in the fixed-point algorithm. Most of the research uses polyhedral sets because they naturally describe linear dynamics and physical limitations \cite{KolmanovskyG98,Kerrigan00}. Ellipsoids are also favored for their close relationship with linear problems and quadratic forms. Using these types of set representations, controlled invariant sets computation can be characterized as linear programming. To circumvent the difficulty of finite termination, approximations of invariant sets are sought instead. The work in \cite{RakovicKKM05} computes invariant outer approximations of the minimal robust invariant sets for linear systems with asymptotically stable dynamics. To obtain an invariant inner approximation of the maximal controlled invariant set of a controllable linear system, one way is to enforce some contractive property on the system dynamics around a compact and convex set containing the origin (called a C-set) \cite{Blanchini94}. An alternative is to compute the null-controllable sets \cite{GutmanC86}, i.e., the set of states that can be controlled to the origin in finite time. 
The recent work \cite{RunggerT16} computes outer and inner approximations of robust controlled invariant sets that do not necessarily contain the origin. For switched linear systems, the computation of maximal controlled invariant sets in \cite{DeSantisDB04} relies on computing inner approximations of the maximal controlled invariant sets of the subsystems, which is based on \cite{Blanchini94}. The computation of maximal dwell-time invariant sets for switching systems under dwell-time switching is considered in \cite{Dehghan12}. Invariance control of power electronic systems, which is modeled as switched linear affine systems, is also studied in \cite{FribourgBook13} using invariant inner-approximations of maximal controlled invariant sets.

\emph{Nonlinear systems:} To compute invariant sets for nonlinear systems, the difficulty not only lies in the problem of finite termination, but also in the computation of reachable sets under nonlinear mappings. Lyapunov functions provide a powerful set of tools for studying invariance control problems, since their sublevel sets naturally define positively invariant sets. Construction of Lyapunov functions for nonlinear systems, however, is a difficult problem. Sum-of-squares (SOS) techniques are used in \cite{WloszekFTSP03} to construct and maximize the level sets of controlled Lyapunov functions, which can be seen as inner approximations of the maximal controlled invariant sets. The techniques lead to nonconvex optimization problems, which are affected by optimization initial conditions. For polynomial systems with convex polyhedral invariant candidates given in prior, invariance control synthesis is simplified to a linear programming problem \cite{BenSassiG12}. In \cite{KordaHJ13}, the authors relax the convex condition and propose a method to outer approximate the maximal controlled sets using occupation measures. The resultant outer approximations, however, are not invariant and hence not suitable for invariance control synthesis.

\emph{Abstraction-based methods:} During computation of the maximal controlled invariant sets, such methods avoid handling system dynamics in an infinite state space  by performing fixed-point computation on finite abstractions of the original system. The finiteness of the abstractions guarantees that such computation terminates in finite time. The resultant fixed-point sets, called winning sets, are equivalent (or approximately equivalent) to the maximal controlled invarianet sets of the original system, if bisimilar models (or approximately bisimilar models) \cite{TabuadaBook09} are used as abstractions. Known systems that have bisimilar models are limitted to controllable linear systems \cite{Tabuada03} and approximately bisimilar models require incrementally stable dynamics \cite{GirardPT10,PolaGT08}. Without these assumptions, over-approximations \cite{LiuOTM13} or similar models \cite{ZamaniPMT12} can be used, but they are prone to generating empty winning sets during discrete synthesis, because spurious transitions are introduced. Therefore, a drawback of using abstraction-based methods is the incompleteness of abstractions as the representations of the original systems. The abstraction refinement method proposed in \cite{NilssonO14} can increase the possibility of finding winning sets, but still without completeness guarantee. 

\subsubsection{Interval methods for reachable sets computation}
Interval analysis, or interval computation, refers to the computational methods that use interval arithmetic with the aim to yield rigorous and reliable results. Such methods have been developed since the 1960s and successfully applied in solving different problems \cite{JaulinBook01}, including computing reachable sets for continuous-time systems \cite{ChenSA14} by way of validated numerical solutions to initial value problems for ordinary differential equations \cite{NedialkovJC99}. A major advantage of using interval methods for reachable sets computation is the flexibility to represent any compact set involved in the computation as unions of intervals. Moreover, set approximation error converges through iterative interval refinement technique. This technique, termed as branch and prune approach, is used to solve constraint-satisfaction problems \cite{Granvilliers01,RamdaniN11} and, more recently, to enclose set boundaries \cite{XueSE16}. It also applies in computing pre-images of nonlinear mappings. The corresponding algorithm is known as \emph{Set Inversion Via Interval Analysis} (SIVIA) \cite{JaulinBook01}. A toolbox implementing SIVIA using high-level numerical programming languages is available in \cite{HerreroGTDJ12}.

\subsection{Contributions}
This paper tackles the invariance control problem for discrete-time switched nonlinear systems. Compared with the aforementioned research, the main contributions of this paper are threefold. 

Firstly, using the interval refinement technique, we solve the invariance control problem by computing maximal controlled invariant sets of switched nonlinear systems directly. To the best of our knowledge, this is the first application of interval analysis in solving invariance control problems. With admissible modes recorded during the computation of controlled invariant sets, invariance controllers can be extracted as soon as the computation completes. 
Compared with abstraction-based methods, our method generates an invariance controller with completeness guarantee for switched nonlinear systems without the assumption that the subsystems have common equilibrium points. 
Moreover, the controlled invariant sets are adaptively partitioned according to both specifications and system dynamics, yielding lower computational complexity than abstraction-based methods with uniform grids. Guided by specifications, our approach is similar to the methods proposed in \cite{NilssonHC14,RunggerMT13}. While these works focus on controllable discrete-time linear systems, our results are derived and tested for the general switched nonlinear systems.

Secondly, we develop algorithms to compute outer and inner approximations of the maximal controlled invariant sets of switched nonlinear systems, which terminate in finite time. The outer approximation can be made as close to the exact maximal invariant sets as required, if a sufficiently small precision parameter is used. For computing inner approximations, we show that a robust invariance condition has to be satisfied in order to obtain invariant inner approximations within a finite number of iterations. This condition addresses the relationship between system robustness and the computation of maximal controlled invariant sets for the first time in the literature. Lastly, we prove that there exists an invariance controller taking values with respect to a partition of the controlled invariant set if the same robust invariance condition is satisfied. We believe this is the first result on partition-based invariance control problem for general switched nonlinear systems.

\subsection{Organization and Notation}
The rest of the paper is organized as follows. In Section II, we introduce switched systems and formulate the invariance control problem. In Section III, we characterize (maximal) controlled invariant sets in terms of set limits. Section IV presents our main results on the computation of outer and inner approximations of maximal controlled invariant sets. In Section V, we explicitly construct the invariance controller. In Section VI, we illustrate the efficiency and effectiveness of our method by three examples and compare them with existing work in the literature. 

\emph{\textbf{Notation}}: $\mathbb{Z}$, $\mathbb{R}$, $\mathbb{R}^n$ denote the set of all integral numbers, real numbers, and $n$-dimensional vectors, respectively; $\mathbb{Z}_{\geq 0}$, $\mathbb{R}_{\geq 0}$, and $\mathbb{R}^n_{\geq 0}$ are the corresponding sets that only have the non-negative members (component-wise non-negative for $n$-dimensional vectors); a compact set is called \emph{full} if it is equal to the closure of its interior; $\|\cdot\|$ denotes the infinity norm in $\Real^n$;
given $\varepsilon \in \Real_{\geq 0}$ and $x\in\Real^n$, define $\mathcal{B}_\varepsilon(x):=\{y\in \Real^n \,|\, \|y-x\|\le\varepsilon\}$; 
given $y\in\Real^n$ and $A\subset\Real^n$, the distance from $y$ to $A$ is defined by $\|y\|_A:=\inf_{x\in A}\|y-x\|$, the boundary and interior of $A$ is denoted by $\partial A$ and $\text{int}(A)$, respectively; given two sets $A,B\subset\Real^n$, $B\setminus A:=\set{x\in B\,\vert\, x\not\in A}$; the Minkowski sum of $A$ and $B$ is defined as $A\oplus B:=\{a+b\,\vert\; a\in A, b\in B\}$, and the Pontryagin difference is defined as $A\ominus B:=\{c\in\Real^n \,\vert\;c+b\in A, \forall b\in B\}$; the Hausdorff distance between $A$ and $B$ is given by $h(A,B):=\max\{d_h(A,B),d_h(B,A)\}$, where $d_h(A,B):=\sup_{a\in A}\|x\|_{B}$ denotes the distance from $A$ to $B$; an interval vector (or box) in $\Real^n$ is denoted by $[x]$, where $[x]:=[x_1]\times\cdots\times[x_n]\subset \Real^n$ and $[x_i]=[\underline{x}_i,\overline{x}_i]\subset \Real$ for $i=1,\cdots,n$; $\underline{x}_i$ represents the infimum of $[x_i]$ and $\overline{x}_i$ the supremum; we also write $[x]=[\underline{x},\overline{x}]$, where $\underline{x}=[\underline{x}_1,\cdots,\underline{x}_n]^T$ and $\overline{x}=[\overline{x}_1,\cdots,\overline{x}_n]^T$ the supremum; the width of the interval $[x]$ is defined as $w([x]):=\max_{1\leq i\leq n}\{\overline{x_i}-\underline{x_i}\}$; the set of all interval vectors in $\Real^n$ is denoted by $\mathbb{IR}^n$; given two sets of intervals $X$, $Y$, $X\subset Y$ denotes $\bigcup_{[x]\in X}[x]\subset \bigcup_{[y]\in Y}[y]$.

\section{Invariance Control Problem}
\subsection{Discrete-time switched nonlinear systems}
We consider discrete-time switched nonlinear systems of the form: 
\begin{equation}
x_{k+1}=f_{p_k}(x_k),\quad k\in \mathbb{Z}_{\ge 0},\label{eq:sw}
\end{equation}
where $x_k$, $x_{k+1}\in\Real^n$ denote the system states at time $k$ and $k+1$, respectively, and $p_k\in \M$ is the mode of the system at time $k$. It is assumed that the set of modes $\M$ is finite. The family of functions $\set{f_{p}}_{p\in\M}:\,\mathbb{R}^n \to \mathbb{R}^n$ are assumed to be continuous, and they determine the nonlinear dynamics for all subsystems.

Any infinite sequence in $\M$ defines a \emph{switching signal} for system (\ref{eq:sw}). We denote a particular switching signal by $\sigma:=\set{p_k}_{k=0}^\infty$, where $p_k\in \M$ for all $k\ge 0$. Given a switching signal $\sigma:=\set{p_k}_{k=0}^\infty$ and an initial state $x_0\in \Real^n$, the solution of system (\ref{eq:sw}) is the unique sequence $\set{x_k}_{k=0}^\infty$ in $\Real^n$ such that (\ref{eq:sw}) is satisfied.

\subsection{Problem formulation}
To state the invariance control problem, we define controlled invariant sets as follows.

\begin{definition} \label{def:invsetc} \em
A set $\Omega\subset\Real^n$ is said to be a \emph{controlled invariant set} for system (\ref{eq:sw}) if, for any initial state $x_0\in\Omega$, there exists a switching signal $\sigma$ such that the resulting solution $\set{x_k}_{k=0}^\infty$ of (\ref{eq:sw}) satisfies $x_k\in\Omega$ for all $k\ge 0$.
\end{definition}

A controlled invariant set of system (\ref{eq:sw}) with a single mode (i.e. there is no control variable) can be called a \emph{positively invariant set}, since only positive-time invariance is considered in this paper.

Given a set $\Omega\subset \Real^n$, the primary objective of invariance control is to determine a subset of $\Omega$, from which the system evolutions can never leave $\Omega$ if proper controls are applied. In \cite{Bertsekas72}, this problem is described as an infinite reachability problem. If $\Omega$ is controlled invariant itself, then it is termed as strongly reachable set. The necessary and sufficient condition for $\Omega$ being infinitely reachable is the existence of a strongly reachable set inside $\Omega$ \cite[Proposition 2]{Bertsekas72}. Among all possible strongly reachable subsets inside $\Omega$, it is of interest to determine the maximal one. Likewise, we have the following definition for switched systems.

\begin{definition} \label{def:maxinvc}\em 
  Given a set $\Omega\subseteq \Real^n$, the set $\I(\Omega)$ is said to be the \emph{maximal controlled invariant set} inside $\Omega$ for system (\ref{eq:sw}), if it is controlled invariant and contains all controlled invariant sets inside $\Omega$.
\end{definition}

To formulate the invariance control problem for discrete-time switched nonlinear systems, we limit our scope to the following type of controllers.

\begin{definition}\em
  A \emph{(memoryless) switching controller} of system (\ref{eq:sw}) is a function
  \begin{equation} \label{eq:ctlr}
    c:\,\Real^n\to 2^{\M}. 
  \end{equation}
A (state-dependent) switching signal $\sigma=\set{p_k}_{k=0}^\infty$ is said to conform to a switching controller $c$, if
  \begin{equation} 
    p_k\in c(x_k),\quad \forall k\ge 0,
  \end{equation}
where $\set{x_k}_{k=0}^\infty$ is the resulting solution of (\ref{eq:sw}). 
\end{definition}

In other words, a switching controller maps the current state into a set of modes that are allowed to apply. A switching signal chooses at each time a specific mode that is allowed by the switching controller.

\begin{definition}\em
A switching controller $c$ is said to be an \emph{invariance controller} for system (\ref{eq:sw}) with respect to a given set $\Omega \in\mathbb{R}^n$ if, for some initial state $x_0 \in \Omega$ and any switching signal $\sigma=\set{p_k}_{k=0}^\infty$ that conforms to $c$, the resulting solution $\set{x_k}_{k=0}^\infty$ of (\ref{eq:sw}) stays inside $\Omega$ for all future time, i.e., $x_k\in\Omega$ for all $k\ge 0$.
\end{definition}

Based on the above definitions, the main problem is stated as follows.

\begin{problem}[\textbf{Invariance Control Problem}] \label{pb:invctl}\em
Given a nonempty set $\Omega\subset \Real^n$ for a system (\ref{eq:sw}),
\begin{enumerate}[label=(\roman*)]
\item \label{itm:pb1} compute the maximal controlled invariant set inside $\Omega$;
\item \label{itm:pb2} synthesize an invariance controller for system (\ref{eq:sw}) with respect to $\Omega$.
\end{enumerate}
\end{problem}

\section{Characterization of controlled invariance} \label{sec:mcis}
In this section, we present some basic results in characterizing controlled invariant sets for switched nonlinear systems.

\begin{definition} \label{def:prec}\em
Given a set $\Omega\subset\Real^n$, the \emph{one-step backward reachable set} of $\Omega$ with respect to system (\ref{eq:sw}) is defined by
$$
\pre(\Omega):=\set{x\in\Real^n:\,\exists p\in\M\text{ such that }f_p(x)\in\Omega}. 
$$
\end{definition}

In other words, we have
\begin{equation}
  \label{eq:prec}
  \pre(\Omega)=\bigcup_{p\in\M}f^{-1}_p(\Omega),
\end{equation}
where $f^{-1}_p(\Omega):=\set{x\in\Real^n:\,f_p(x)\in\Omega}$, i.e., the pre-image of $\Omega$ under $f_p$.

By the definition above and continuity of the functions $\{f_p\}$, it is straightforward to check the following properties.
\begin{proposition}\label{prop:pre}\em
  Let $\Omega\subset \Real^n$, and $A\subset B\subset \Real^n$. Then
  \begin{enumerate}[label=(\roman*)]
  \item \label{itm:pre1}if $\Omega$ is closed, $\pre(\Omega)$ is closed;
  \item \label{itm:pre2}$\pre(A)\subset \pre(B)$;
  \end{enumerate}
\end{proposition}
\begin{proof}
  First of all, we show \ref{itm:pre1}. Let $p\in\M$, and $\{x_k\}_{k=0}^{\infty}$ be a convergent sequence in the set $f^{-1}_p(\Omega)$ with the limit $x^*$, i.e., $\lim_{k\to\infty}x_k=x^*$. By Definition \ref{def:prec}, there is a point $\tilde{x}_k=f_p(x_k)\in \Omega$ for all $k\in \Z_{\geq 0}$, resulting in a sequence $\{\tilde{x}_k\}_{k=0}^\infty$. Given $\Omega$ is a closed set, it follows that there is a convergent subsequence $\{\tilde{x}_{j_k}\}$ such that $\lim_{k\to\infty}\tilde{x}_{j_k}=\tilde{x}^*\in \Omega$. With the continuity of $f_p$, we have
$$\tilde{x}^*=\lim_{k\to\infty}\tilde{x}_{j_k}=\lim_{k\to\infty}f_p(x_{j_k})=\lim_{\substack{x\to x^*}}f_p(x)=f_p(x^*),$$
which means $x^*$ is inside $f_p^{-1}(\Omega)$, and thus $f_p^{-1}(\Omega)$ is closed. Since $\pre(\Omega)$ is a finite union of $f^{-1}_p(\Omega)$, $\pre(\Omega)$ is closed.

Considering \ref{itm:pre2}, for any $x\in \pre(A)$, there exists a mode $p\in \M$ such that $f_p(x)\in A\subset B$, and by definition, $x$ also belongs to $\pre(B)$. This proves $\pre(A)\subset \pre(B)$.
\end{proof}

\begin{proposition}\label{prop:invc}\em
A set $\Omega\subset\Real^n$ is controlled invariant for system (\ref{eq:sw}) if and only if $\Omega\subset \pre(\Omega)$.
\end{proposition}

\begin{proof}
If $\Omega\subset \pre(\Omega)$, then, for any $x_0\in\Omega\subset \pre(\Omega)$, there exists $p_0\in\M$ such that $f_p(x_0)\in\Omega$. Continuing this indefinitely will give a switching signal $\set{p_k}_{k=0}^\infty$ such that the resulting solution $\set{x_n}_{n=k}^\infty$ of (\ref{eq:sw}) satisfies $x_n\in\Omega$ for all $n\ge k$. Hence $\Omega$ is controlled invariant.

If $\Omega$ is controlled invariant, by definition, for any $x_0\in\Omega$, there exists a switching signal $\sigma$ such that the resulting solution $\set{x_k}_{k=0}^\infty$ of (\ref{eq:sw}) satisfies $x_k\in\Omega$ for all $k\ge 0$. In particular, $x_1=f_{p_0}(x_0)\in\Omega$. This shows $x_0\in \pre(\Omega)$. Hence $\Omega\subset \pre(\Omega)$.
\end{proof}

To characterize maximal invariant sets, we define a mapping $I:X\to X$ as follows:
\begin{equation}
  \label{eq:mapi}
  I(X)=\pre(X)\cap X.
\end{equation}

Given a set $\Omega\subset \Real^n$, let $I^j$ ($j\in \Z_{>0}$) denote the $j$-times repeated compositions of the mapping $I$, then $I^j(\Omega)$ is the set of states from which the solutions of system (\ref{eq:sw}) can stay in $\Omega$ for $j$ steps of time. Letting $j=\infty$ gives the limit set $\lim_{j\to \infty}I^j(\Omega)$. Presumably, this limit set represents the maximal controlled invariant set inside $\Omega$. The following proposition formalizes this result for system (\ref{eq:sw}). A similar characterization of maximal controlled invariant set for non-switched systems can be found in \cite{Bertsekas72}.

\begin{proposition}\label{prop:conv}\em
Suppose that $\Omega\subset \Real^n$ is compact. Then 
  \begin{enumerate}[label=(\roman*)]
   \item \label{itm:fp1}$\I(\Omega)=\lim_{n\to\infty}I^n(\Omega)$;
   \item \label{itm:fp2}$\I(\Omega)$ is compact;
   \item \label{itm:fp3}$\I(\Omega) = I(\I(\Omega))$.
  \end{enumerate}
\end{proposition}
\begin{proof}
\ref{itm:fp1} First, it is straightforward to check that $I^n(\Omega)$ is a monotonically decreasing sequence of sets in the sense that $I^{n+1}(\Omega)\subset I^{n}(\Omega)$ for all $n\ge 1$. By Proposition \ref{prop:pre}, $I^n(\Omega)$ is a closed set for all $n\ge 1$. Thus the set limit of $I^n(\Omega)$ exists and is given by $\bigcap_{n=1}^{\infty}I^n(\Omega)$. 
Second, it is easy to check that $\I(\Omega)\subset I^n(\Omega)$ for all $n\ge 1$ by induction, using the fact that $\I(\Omega)$, if nonempty, is a controlled invariant subset of $\Omega$ and Proposition \ref{prop:invc}. Hence, $\I(\Omega)\subset \bigcap_{n=1}^{\infty}I^n(\Omega)=\lim_{n\ra\infty}I^n(\Omega)$.

Third, we claim that $\lim_{n\ra\infty}I^n(\Omega)\subset \I(\Omega)$. If $\lim_{n\ra\infty}I^n(\Omega)$ is empty, this trivially holds. If not, pick any $x_0\in \lim_{n\ra\infty}I^n(\Omega)$. Then $x_0\subset I^n(\Omega)$ for all $n\ge 1$. It follows that there exists $p_n\in \M$ such that $f_{p_n}(x_0)\in I^{n-1}(\Omega)$ for all $n\ge 1$. Since $\M$ is finite, the sequence $\set{p_n}_{n=1}^\infty\subset \M$ must admit a constant subsequence. In other words, there exists $p\in\M$ such that $f_p(x_0)\subset I^{n-1}(\Omega)$ for infinitely many $n\ge 1$. By monotonicity of the sequence $I^n(\Omega)$, this implies $f_p(x_0)\in \bigcap_{n=1}^{\infty}I^n(\Omega)=\lim_{n\ra\infty}I^n(\Omega)$. Hence $\lim_{n\ra\infty}I^n(\Omega)$ is a controlled invariant subset of $\Omega$. By Definition \ref{def:maxinvc}, $\lim_{n\ra\infty}I^n(\Omega)\subset \I(\Omega)$. Hence, $\I(\Omega)=\lim_{n\ra\infty}I^n(\Omega)$.

\ref{itm:fp2} By Definition \ref{def:maxinvc}, $\I(\Omega)\subset \Omega$, and thus $\I(\Omega)$ is bounded. By Proposition \ref{prop:pre} \ref{itm:pre1}, $I(X)=\pre(X)\cap X$ is closed if $X$ is closed, which implies $I^j(\Omega)$ is closed for all $j\in \Z_{\geq 0}$. By \ref{itm:fp1},  $\I(\Omega)=\bigcap_{n=1}^{\infty}I^n(\Omega)$ is closed, because infinite intersection of closed set is still closed. Hence, $\I(\Omega)$ is closed and bounded, which means it is compact.

\ref{itm:fp3} By (\ref{eq:mapi}), $I(\I(\Omega))=\pre(\I(\Omega))\cap \I(\Omega)$, and thus $I(\I(\Omega)) \subset \I(\Omega)$. Assume $x\in \I(\Omega)$. Then $x\in\pre(\I(\Omega))$, otherwise there is no mode $p\in \M$ such that $f_p(x)\in \I(\Omega)$. Hence, $\I(\Omega)\subset I(\I(\Omega))$, which completes the proof.
\end{proof}

Based on Proposition \ref{prop:conv} \ref{itm:fp3}, $\I(\Omega)$ is a fixed point of $I$, and the following well-known fixed-point algorithm can be used to compute maximal controlled invariant sets for system (\ref{eq:sw}).

\begin{algorithm}[H]
  \centering
  \caption{Computation of $\I(\Omega)$}
  \label{alg:mcis0}
  \begin{algorithmic}[1]
    \Require $\Omega$
    \State $\widetilde{X}=\Omega, X= \varnothing$
    \While{$\widetilde{X} \neq X$}
    \State $X=\widetilde{X}, \widetilde{X}=\pre(X) \cap X$
    \EndWhile
    \State \Return $X$
  \end{algorithmic}
\end{algorithm}

In practice, there are two major challenges in computing $\I(\Omega)$. First, it is very difficult, if not impossible, to exactly compute the backward reachable set $\pre(X)$ because of the nonlinear mappings $\set{f_p}_{p\in \M}$. Even for linear systems with polyhedral or ellipsoidal constraints, set operations such as Pontryagin difference are likely to introduce irregular shapes, which makes computation of accurate reachable sets impossible. Therefore, one has to seek approximations of $\pre(X)$. 

Second, Algorithm \ref{alg:mcis0} may not terminate in a finite number of steps. We call $\I(\Omega)$ \emph{finitely determined} if Algorithm \ref{alg:mcis0} stops in a finite number of steps. 
Only for some special cases like unconstrained controllable LTI systems and finite state machines with bounded constraints, $\I(\Omega)$ can be finitely determined \cite{VidalSLS00}. 
We will address these two challenges in the next section as the main results of our paper.

\section{Computation of controlled invariant sets}

This section provides the technical details on computing the controlled invariant sets.

\subsection{Interval approximation of backward reachable sets}
Computing the one-step backward reachable sets for system (\ref{eq:sw}) relies on the computation of the pre-image $f^{-1}_p(\Omega)$ for each $p\in\M$. In this paper, we use interval computation to approximate these pre-images, not only for its simplicity, but also for its convergence guarantee under mild assumptions. Following the branch-and-prune approach \cite{JaulinBook01,Granvilliers01}, we present Algorithm \ref{alg:cpre} for computing the image of the nonlinear mapping $I$. 

Fundamental to computing interval images is the concept of convergent inclusion functions.

\begin{definition}\em \cite{JaulinBook01} \label{def:intf}
Consider a function $f:\,\Real^n\rightarrow\Real^m$. The corresponding interval function $[f]:\,\mathbb{IR}^n\rightarrow\mathbb{IR}^m$ is called a \emph{convergent inclusion function} of $f$ if the following two conditions hold:
\begin{enumerate}[label=(\roman*)]
\item \label{itm:intf1}$f([x])\subset [f]([x])$ for all $[x]\in \mathbb{IR}^n$; 
\item \label{itm:intf2}$\lim_{w([x])\to 0}w([f]([x]))=0$.
\end{enumerate}
\end{definition}

For a vector-valued function $f$, its convergent inclusion function is not unique. Methods vary in obtaining such inclusion functions. One can compute the infimum and supremum of $f([x])$ by performing optimizations on the interval $[x]$ if they are trivial. One straightforward inclusion function is called the natural inclusion function, which is the result of replacing the variables and operations in a function by their interval counterparts. The natural inclusion function is known to have at least a linear convergence rate. For higher precision, centered-form and mean-value form can also be used \cite{JaulinBook01}.

\begin{algorithm}[htbp]
  \caption{Approximation of $I(\Omega)$}
  \label{alg:cpre}
  \begin{algorithmic}[1]
    \Procedure{CPre}{$[f_p]_{p\in \M},\Omega,\varepsilon$}
    \State $\underline{X}\leftarrow\varnothing,\Delta X\leftarrow\varnothing, X_c\leftarrow\varnothing$
    \State $\C\leftarrow \varnothing, List\leftarrow\{[\Omega]\}$
    
    \While{$List\neq \varnothing$}
    \State $[x]\leftarrow Pop(List)$
    \State $C\leftarrow \varnothing$
    \If{$[f_p]([x])\cap \Omega=\varnothing$ for all $p\in\M$}
    \State $X_c\leftarrow X_c\cup [x]$

    \ElsIf{$[f_p]([x]) \subset \Omega$ for some $p\in\M$}
    \State $\underline{X}\leftarrow \underline{X}\cup [x]$;
    \State $C\leftarrow C\cup p$

    \Else

    \If{$w([x])<\varepsilon$}
    \State $\Delta X\leftarrow \Delta X\cup [x]$;
    \Else
    \State $\{L[x],R[x] \}=Bisect([x])$
    \State $Push(List, \set{L[x],R[x]}$
    \EndIf

    \EndIf
    
    \If{$C\neq \varnothing$}
    \State $\C\leftarrow \C\cup C$
    \EndIf

    \EndWhile

    \Return $\underline{X}, \Delta X, X_c, \C$
    \EndProcedure
  \end{algorithmic}
\end{algorithm}

The algorithm takes as input a compact set $\Omega$, which is assumed to be an interval or a finite union of intervals. This is without loss of generality, because any compact set can be arbitrarily approximated by a union of intervals. At each iteration, Algorithm \ref{alg:cpre} checks if the image $[f_p]([x])$ of a particular box $[x]$ is contained in $\Omega$ for some $p\in\M$, or completely outside of $\Omega$ for any $p\in\M$. If neither, and the box size is greater than $\varepsilon$, then $[x]$ is deemed to be undetermined and divided into two subintervals $L[x]$ and $R[x]$ by bisection, which are given by
\begin{equation*}
  \begin{aligned}
    L[x]&=[\underline{x}_1,\overline{x}_1] \times\cdots\times [\underline{x}_j,(\underline{x}_j+\overline{x}_j)/2] \times\cdots\times [\underline{x}_n,\overline{x}_n],\\
    R[x]&=[\underline{x}_1,\overline{x}_1] \times\cdots\times [(\underline{x}_j+\overline{x}_j)/2,\overline{x}_j] \times\cdots\times [\underline{x}_n,\overline{x}_n],
  \end{aligned}
\end{equation*}
where $j$ is the dimension in which the box $x$ attains its width. A box will not go through subdivision once its size is less than $\varepsilon$, which is used to control the smallest size of the intervals, and thus control the precision of the set $\pre(\Omega)$. The list of intervals that entirely belong to $\pre(\Omega)$ is denoted by $\underline{X}$ while the list of those that are mapped outside of $\Omega$ by $[f_p]$ for any $p\in\M$ is denoted by $X_c$. The intervals that are partly inside $\pre(\Omega)$, i.e., undetermined intervals, are collected in $\Delta X$. The list $\C$ is a list of modes that render the system (\ref{eq:sw}) invariant within $\Omega$ for the corresponding intervals in $\underline{X}$. 

To analyze the precision of pre-image computation with a given precision parameter $\varepsilon$, we rely on the following assumption.

\begin{assumption}\em \label{asp:f}
  Let $\Omega\subset\Real^n$. There exists a $\rho_1>0$ such that system (\ref{eq:sw}) satisfies the following condition for all $p\in\M$:
  \begin{equation}\label{eq:lipschitz}
      \|f_p(x)-f_p(y)\|\leq \rho_1\|x-y\|,\quad \forall x,y\in\Omega.
  \end{equation}
\end{assumption}

This is essentially a local compactness assumption on $f_p$ for all $p\in\M$. If $f_p$ is continuously differentiable on $\Omega$ for all $p\in\M$, and $\Omega$ is compact, then $\rho_1=\max_{x\in\Omega,p\in\M}\|J_xf_p\|$, where $J_x$ is the Jacobian matrix at $x$. Based on (\ref{eq:lipschitz}), it is possible to choose an inclusion function $[f_p]$ for each $f_p$ (e.g., mean-value form \cite{JaulinBook01}) such that 
  \begin{equation}\label{eq:lipschitz2}
w([f_p]([x]))\leq \rho_1 w([x]),\quad \forall [x]\in\mathbb{IR}^n. 
  \end{equation}

\begin{theorem}\em \label{prop:cpre}
  Suppose that $\Omega$ and $X=I(\Omega)$ are compact and full. Let $\overline{X}:=\underline{X}\cup \Delta X$, where $\underline{X}$ and $\Delta X$ are outputs of Algorithm \ref{alg:cpre}. If Assumption \ref{asp:f} holds on $\Omega$, then
  \begin{equation*}
    I(\Omega\ominus\mathcal{B}_{\rho_1\varepsilon})\subset \underline{X}\subset X\subset \overline{X}\subset I(\Omega\oplus\mathcal{B}_{\rho_1\varepsilon}).
  \end{equation*}
\end{theorem}

\begin{proof}
  It is clear from Algorithm \ref{alg:cpre} that $\underline{X}\subset X\subset \overline{X}\subset \Omega$. This algorithm only stops when the widths of the undetermined intervals are less than the given accuracy parameter, which means that for all $[x]\in \Delta X$, $w([x])<\varepsilon$.

By Assumption \ref{asp:f} and (\ref{eq:lipschitz2}), we have $w([f_p]([x]))\leq \rho_1 w([x])<\rho_1\varepsilon$. For any $[x]\in \Delta X$, there exists a $p\in\M$ such that $[f_p]([x])\cap \Omega\neq \varnothing$. By the definition of Minkowski sum, it follows that $[f_p]([x])\subset (\Omega\oplus\mathcal{B}_{\rho_1\varepsilon})$. Also, for any $[x]\in \underline{X}$, we can find a $p\in \M$ such that $[f_p]([x])\subset \Omega$. Hence, $\overline{X}=(\underline{X}\cup \Delta X) \subset (\pre(\Omega\oplus\mathcal{B}_{\rho_1\varepsilon})\cap\Omega)\subset I(\Omega\oplus\mathcal{B}_{\rho_1\varepsilon})$.

To prove the second half, we aim to show that for all $y\in (\Omega\ominus\mathcal{B}_{\rho_1\varepsilon})$, there exists a $ p\in\M$ such that $y\in\bigcup_{[x]\in\underline{X}}f_p([x])$. Suppose this is not true. Then it is fair to say that there exists a $p\in\M$ such that $y\in \bigcup_{[x]\in \Delta X}f_p([x])$, otherwise $y\in \bigcup_{p\in\M} f_p(X_c)$, which means $y\notin \Omega$. It follows that there exists $y'\in\partial \Omega$ such that $\|y'-y\|=\gamma<\rho_1\varepsilon$. Then for any given $\delta>0$ satisfying $\delta+\gamma<\rho_1\varepsilon$, there exists a point $y''\in (\Real^n\setminus \Omega)$ such that $\|y''-y'\|\leq \delta$. Thus, $\|y''-y\|\leq \|y''-y'\|+\|y'-y\|\leq \delta+\gamma<\rho_1\varepsilon$. This implies there exists a point $z\in\mathcal{B}_{\rho_1\varepsilon}$ such that $z+y=y''\notin \Omega$, which means $y\notin (\Omega\ominus \mathcal{B}_{\rho_1\varepsilon})$. This is a contradictory with the hypothesis. Hence, $\underline{X}\supset I(\Omega\ominus\mathcal{B}_{\rho_1\varepsilon})$. 
\end{proof}

Furthermore, the sets $\underline{X}$ and $\overline{X}$ converge to the exact set $X$ if $f_p$ (for all $p\in\M$) is invertible on $\Omega$. This implies that the exact set $X$ can be approximated with arbitrary precision. Before presenting the convergence result, we provide the following proposition based on the basic properties of the Pontryagin difference \cite{KolmanovskyG98}.

\begin{proposition} \label{prop:hdis}\em
Let $\Omega \subset \Real^n$ be a compact set, and $\rho>0$. Assume $\Omega\ominus\mathcal{B}_{\rho}\neq \varnothing$. Then $h(\Omega,\Omega\ominus\mathcal{B}_{\rho})\geq \rho$. Furthermore,
\begin{enumerate}[label=(\roman*)]
\item \label{itm:h1}If $\Omega$ is an interval, then $h(\Omega,\Omega\ominus\mathcal{B}_{\rho})=\rho$; 
\item \label{itm:h2}if $\Omega$ is full, then $\lim_{\rho\to 0}h(\Omega,\Omega\ominus\mathcal{B}_{\rho})=0$.
\end{enumerate}
\end{proposition}

\begin{proof}
  We show that $h(\Omega,\Omega\ominus\mathcal{B}_{\rho})\geq \rho$ first. Suppose that this is not true. Then $h(\Omega,\Omega\ominus\mathcal{B}_{\rho})< \rho$. It follows that $\forall v\in \Omega$, $\|v\|_{\Omega\ominus\mathcal{B}_{\rho}}<\rho$. Since $\Omega$ and $\Omega\ominus\mathcal{B}_\rho$ is compact, for a point $v'\in \partial\Omega$, we can find a $u_m\in \Omega\ominus\mathcal{B}_\rho$, such that $\|u_m-v'\|=\|v'\|_{\Omega\ominus\mathcal{B}_{\rho}}=\gamma<\rho$. By the property of boundary point, for any $\delta>0$ satisfying $\delta+\gamma<\rho$, there is a point $z\notin \Omega$ such that $\|z-v'\|\leq \delta$. Hence, $\|z-u_m\|\leq \|z-v'\|+\|v'-u_m\|\leq \delta+\gamma<\rho$. This implies $u_m\notin \Omega\ominus\mathcal{B}_\rho$, which is a contradiction. Thus, $h(\Omega,\Omega\ominus\mathcal{B}_{\rho})\geq \rho$.

  If $\Omega$ is an interval, it is clear that $h(\Omega,\Omega\ominus\mathcal{B}_{\rho})=\rho$. Hence, we only need to show that (ii) holds.

Since $\Omega\subset\Real^n$ is compact and full, every point in $\Omega$ is the limit of a convergent subsequence from int$(\Omega)$. Let $\{x_n\}$ be such sequence. Then given any $\varepsilon>0$ for a point $x\in \Omega$, there exists a positive integer $N$ such that
$$\|x_n-x\|<\varepsilon, \forall n\geq N,$$
where $x_n\in \text{int}(\Omega)$. It follows that there exists $\delta>0$ such that $\mathcal{B}_\delta(x_n)\subset \Omega$, which implies $x_n\in \Omega\ominus\mathcal{B}_\delta$. Thus
$$\lim_{\delta\to 0}\inf_{y\in\Omega\ominus\mathcal{B}_\delta}\|y-x\|=0.$$
Since it holds for all the points in $\Omega$, we have
\begin{align*}
  \lim_{\rho\to 0}d_h(\Omega,\Omega\ominus\mathcal{B}_{\rho})&=\lim_{\rho\to 0}\sup_{x\in \Omega}\|x\|_{\Omega\ominus\mathcal{B}_{\rho}}\\
&=\sup_{x\in\Omega}\lim_{\rho\to 0}\|x\|_{\Omega\ominus\mathcal{B}_{\rho}}=0.
\end{align*}
Therefore,
\begin{equation*} 
\lim_{\rho\to 0}h(\Omega,\Omega\ominus\mathcal{B}_{\rho})= \lim_{\rho\to 0} d_h(\Omega,\Omega\ominus\mathcal{B}_{\rho})=0.
\end{equation*}
\end{proof}

\begin{proposition}\em
  Let the assumptions of Theorem \ref{prop:cpre} hold. If, in addition, $f_p$ is invertible and the inverse function $f_p^{-1}$ is Lipschitz continuous on $\Omega$ for all $p\in\M$, then
  \begin{equation*}
    \lim_{\varepsilon\to \infty}h(X,\underline{X})=0, \lim_{\varepsilon\to \infty}h(X,\overline{X})=0,
  \end{equation*}
where $\varepsilon$ is the precision parameter of Algorithm \ref{alg:cpre}.
\end{proposition}

\begin{proof}
  Denote by $\rho_2$ the maximum Lipschitz constant of $\{f_p^{-1}\}_{p\in\M}$. For any $y\in \partial \Omega$, there exists a $p\in\M$ such that $f_p^{-1}(y)\in \partial X$. Denote $x=f_p^{-1}(y)$. If this is not true, then $x\in \text{int}(\Omega)$. Thus there exists $\delta_x>0$ such that $\mathcal{B}_{\delta_x}(x) \subset X$. By the continuity of $f_p^{-1}$, there exists a $\delta_y>0$ such that $\forall y'\in\mathcal{B}_{\delta_y}(y)$, $f_p^{-1}(y')\in\mathcal{B}_{\delta_x}(x)$. Thus $f_p^{-1}(y')\in X$. This implies $y \in \text{int}(\Omega)$, which is a contradiction. Any point $x\in [x]\in \Delta X$, there exists a $y'\in \partial \Omega$, such that $\|y'-f(x)\|<\rho_1\varepsilon$. By Theorem \ref{prop:cpre}, $x'=f_p^{-1}(y')\in\partial X$. It follows that $\|x'-x\|=\|f_p^{-1}(y')-f_p^{-1}(y)\|\leq \rho_2\|y'-y\|<\rho_2\rho_1\varepsilon$. Hence, for all $x\in [x]\in \Delta X$, there exists a $x'\in \partial X$, such that $\|x'-x\|<\rho_2\rho_1\varepsilon$. Using the same argument as in proving Theorem \ref{prop:cpre}, we have $\overline{X}\subset (X\oplus \mathcal{B}_{\rho_2\rho_1\varepsilon})$ and $\underline{X}\supset (X\ominus \mathcal{B}_{\rho_2\rho_1\varepsilon})$. Hence, $h(X, \underline{X})\leq h(X, X\ominus\mathcal{B}_{\rho_2\rho_1\varepsilon})$ and $h(X, \overline{X})\leq h(X, X\oplus\mathcal{B}_{\rho_2\rho_1\varepsilon})$. By Proposition \ref{prop:hdis} \ref{itm:h2} and the definition of Minkowski sum, the conclusion is proved.
\end{proof}

\subsection{Outer approximations of maximal controlled invariant sets}
As illustrated in section \ref{sec:mcis}, computing the exact maximal controlled invariant sets involves possibly an infinite number of iterations. 
It is natural to seek finitely determined outer approximations of the maximal invariant sets \cite{KordaHJ13,RunggerT16}. 

The following Algorithm \ref{alg:outerc}, which can be seen as a concrete realization of the conceptual Algorithm \ref{alg:mcis0}, generates an outer approximation of the maximal controlled invariant set within a given set $\Omega$ for system (\ref{eq:sw}).

\begin{algorithm}[htbp]
  \caption{Outer Approximation of $\I(\Omega)$}
  \label{alg:outerc}
  \begin{algorithmic}[1]
  \Require $\set{[f_p]}_{p\in\P},\Omega, \varepsilon$
    \State $X\leftarrow [\Omega],Y\leftarrow\Omega, X_c\leftarrow[\Omega]$
    \While{$X_c\neq \varnothing$}

    \State $[\underline{X}, \Delta X, X_c, \C]=\text{CPre}(\set{[f_p]}_{p\in\P},Y,X,\varepsilon)$
    \State $\overline{X}\leftarrow \underline{X}\cup\Delta X$
    \State $Y\leftarrow \cup_{[x]\in\overline{X}}[x]$
    \State $X\leftarrow \overline{X}$
    
    \EndWhile	
    \State  \Return $Y,\C$
   \end{algorithmic}
\end{algorithm}

\begin{theorem}\em \label{prop:outer}
  Let $\Omega\subset \Real^n$ be compact. Suppose that Assumption \ref{asp:f} holds. Denote by $\overline{Y}^\varepsilon$ the output of Algorithm \ref{alg:outerc} for a given precision $\varepsilon$. Then Algorithm \ref{alg:outerc} terminates in a finite number of steps. Furthermore, $\overline{Y}^\varepsilon$ is an union of intervals satisfying the following properties:
  \begin{enumerate}[label=(\roman*)]
  \item \label{itm:outer1}$I(\overline{Y}^\varepsilon)\subset \overline{Y}^\varepsilon\subset I(\overline{Y}^\varepsilon\oplus\mathcal{B}_{\rho_1\varepsilon})$ 
;
  \item \label{itm:outer2}if $0<\varepsilon_1<\varepsilon_2$, $\I(\Omega)\subset \overline{Y}^{\varepsilon_1}\subset \overline{Y}^{\varepsilon_2}$;
  \item \label{itm:outer3}$\I(\Omega)=\lim_{\varepsilon\to 0}\overline{Y}^\varepsilon$.
  \end{enumerate}
\end{theorem}
\begin{proof}
  We use subscript $j$ to denote the corresponding sets in $j$th iteration ($j\in \Z_{\geq 0}$), i.e., $X_{c,j}$, $\Delta X_{j}$, and $Y_j$ represent $X_c$, $\Delta X$ and $Y$ in the $j$th iteration, respectively.

  Firstly, we show that Algorithm \ref{alg:outerc} stops in finite steps. In each iteration, the part taken by $X_{c,j}$ is removed from $Y_j$, yielding $Y_j\subset Y_{j-1}$. Hence, $\{Y_j\}$ is an non-increasing sequence of sets with $Y_0=\Omega$. Under a given precision $\varepsilon>0$, $Y_j$ is represented by a union of intervals with minimum width $\varepsilon$. Suppose for all $j\in \Z_{\geq 0}$, $X_{c,j}\neq \varnothing$. This implies there always be some intervals removed from $Y_{j-1}$, and $\{Y_j\}$ is strictly decreasing. Then there must exists an $N\in \Z_{>0}$ such that $Y_N=\varnothing$, which means $X^N_c=\varnothing$ and the algorithm stops at step $N$, since $Y_j$ is comprised of a finite number of intervals ($\Omega$ is compact). Therefore, Algorithm \ref{alg:outerc} terminates in finite steps.

Secondly, we are going to prove the correctness of \ref{itm:outer1}. Assume Algorithm \ref{alg:outerc} stops at step $N>0$. Then $X_{c,N}=\varnothing$, and $Y_{N-1}=X_N\cup \Delta X_N= Y_N=\overline{Y}^\varepsilon$. 
By Theorem \ref{prop:cpre}, $Y_N\subset I(Y_{N-1}\oplus \mathcal{B}_{\rho_1\varepsilon})$. 
Moreover, we can conclude that $x\in Y_N$ for all $x\in I(Y_N)$, otherwise $[f_p](x)\cap Y_{N-1}=\varnothing$ for all $p\in \M$, i.e., $x\notin I(Y_{N})=I(Y_{N})$. Hence, $I(Y_{N-1})\subset Y_N$, and \ref{itm:outer1} is proved.

Thirdly, to prove \ref{itm:outer2}, we first consider $\I(\Omega)\subset Y^\varepsilon$. For the sake of contradiction, let $y\in \I(\Omega)$ but $y\notin Y_N$ for a $N\in\Z_{>0}$. Then $y\in \Omega \setminus Y_N$. According to the algorithm, $\forall z\in \Omega\setminus Y_N$, there must be a step $0<j\leq N$ such that $[f_p](z)\cap \Omega=\varnothing$ for all $p\in \P$. This indicates that $z\notin \I(\Omega)$, which is a contradiction. Thus $\I(\Omega)\subset Y^{\varepsilon}$. Next we prove $Y^{\varepsilon_1}\subset Y^{\varepsilon_2}$ by induction. Consider the first two steps: $Y_0^{\varepsilon_1}=Y_0^{\varepsilon_2}=\Omega$. Since $0<\varepsilon_1<\varepsilon_2$, some intervals in $\Delta X_1^{\varepsilon_2}$ will be divided into finer boxes and are possible to be included in $X_{c,1}^{\varepsilon_1}$, and thus $X_{c,1}^{\varepsilon_2}\subset X_{c,1}^{\varepsilon_1}$. Together with $Y_1^{\varepsilon_1}=Y_0^{\varepsilon_1}\setminus X_{c,1}^{\varepsilon_1}$, and $Y_1^{\varepsilon_2}=Y_0^{\varepsilon_2}\setminus X_{c,1}^{\varepsilon_2}$, we have $Y_1^{\varepsilon_1}\subset Y_1^{\varepsilon_2}$. Assume $Y_j^{\varepsilon_1}\subset Y_j^{\varepsilon_2}$ for any step $1\leq j<N$. Then $X_{c,j}^{\varepsilon_2}\subset X_{c,j}^{\varepsilon_1}$, which gives $Y_{j+1}^{\varepsilon_1}\subset Y_{j+1}^{\varepsilon_2}$. Hence \ref{itm:outer2} is proved.

Lastly, we show \ref{itm:outer3}. Consider a decreasing sequence $\{\varepsilon_j\}_{j=1}^{\infty}$ with $\varepsilon_j>0$ and $\lim_{j\to \infty}\varepsilon_j=0$. Since $\overline{Y}^{\varepsilon_j}$ is compact, $\lim_{\varepsilon_j\to 0}\overline{Y}^{\varepsilon_j}$ exists and is given by the compact set $\bigcap_{j=1}^{\infty}\overline{Y}^{\varepsilon_j}$. Let $\overline{Y}=\bigcap_{j=1}^{\infty}\overline{Y}^{\varepsilon_j}$. If every $\overline{Y}^{\varepsilon_j}$ is nonempty, then $\overline{Y}$ is nonempty. By \ref{itm:outer2}, $\I(\Omega)\subset \overline{Y}^{\varepsilon_j}$ for all $j\geq 1$. Then it is clear that $\I(\Omega)\subset \overline{Y}$. Next, we claim that $\overline{Y}\subset \I(\Omega)$. If this is not true, then there exists $y\in \overline{Y}$ such that $f_p(y)\notin \overline{Y}$ for all $p\in \P$, i.e., $f_p(X)\in \overline{Y}_c$, which is the complement of $\overline{Y}$ and is open. Then it follows that there exists $\delta>0$ such that $\mathcal{B}_{\delta}(f_p(y))\subset \overline{Y}_c$ for all $p\in\P$. Furthermore, by the definition of set limit, there exists a $J_1$ sufficiently large such that $B_{\delta}(f_p(y))\cap \overline{Y}^{\varepsilon_j}=\varnothing$ for all $p\in\P$ and $j\geq J_1$. Then it is only possible that $y\in [x]\in \Delta X_j, j\geq J_1$. Since $f_p$ is a continuous inclusion function, there exists a $J_2$ such that $[f_p]([x])\subset \mathcal{B}_{\delta}(f_p(y))$ for all $p\in\P$ and $[x]\in \Delta X_j, j\geq J_2$. Then for all $j\geq \max\{J_1, J_2\}$, we have $[f_p]([x])\cap Y_N=\varnothing$, which is contradictory with the fact that $y\in \Delta X_j$. Hence, \ref{itm:outer3} is true.
\end{proof}

Theorem \ref{prop:outer} indicates that the exact maximal invariant sets can be outer approximated in an arbitrary precision, as illustrated in the following example.

\begin{example}\em \label{eg:lti}
  Consider a single mode linear time invariant system $x_{k+1}=Ax_k$, where
  \begin{equation*}
    A=\begin{bmatrix} 1.0810 & 0.4517 \\ -0.0903 & 0.7197 \end{bmatrix}.
  \end{equation*}

With a pair of complex eigenvalues $0.9003\pm 0.0903i$, the region of attraction of this LTI system is the entire plane. Given a compact set $\Omega=[-1,1]\times[-1,1]$, however, the positively invariant set inside $\Omega$ is not simply $\Omega$ itself because of the spiral trajectories governed by the dynamics.

In this example, the maximal positively invariant set within $\Omega$ is bounded by two trajectories, which is marked by the two red curves. Figure \ref{fig:LTIouters} shows the approximation results with different choices of precision $\varepsilon$ ($\varepsilon=0.05,0.01,0.0063,0.001$, respectively) using Algorithm \ref{alg:outerc}. It can be observed that the approximation error decreases as $\varepsilon$ becomes smaller.%Time cost are $4.018s, 27.856s, $

\begin{figure}[htbp]
  \centering
  \includegraphics[scale=0.65]{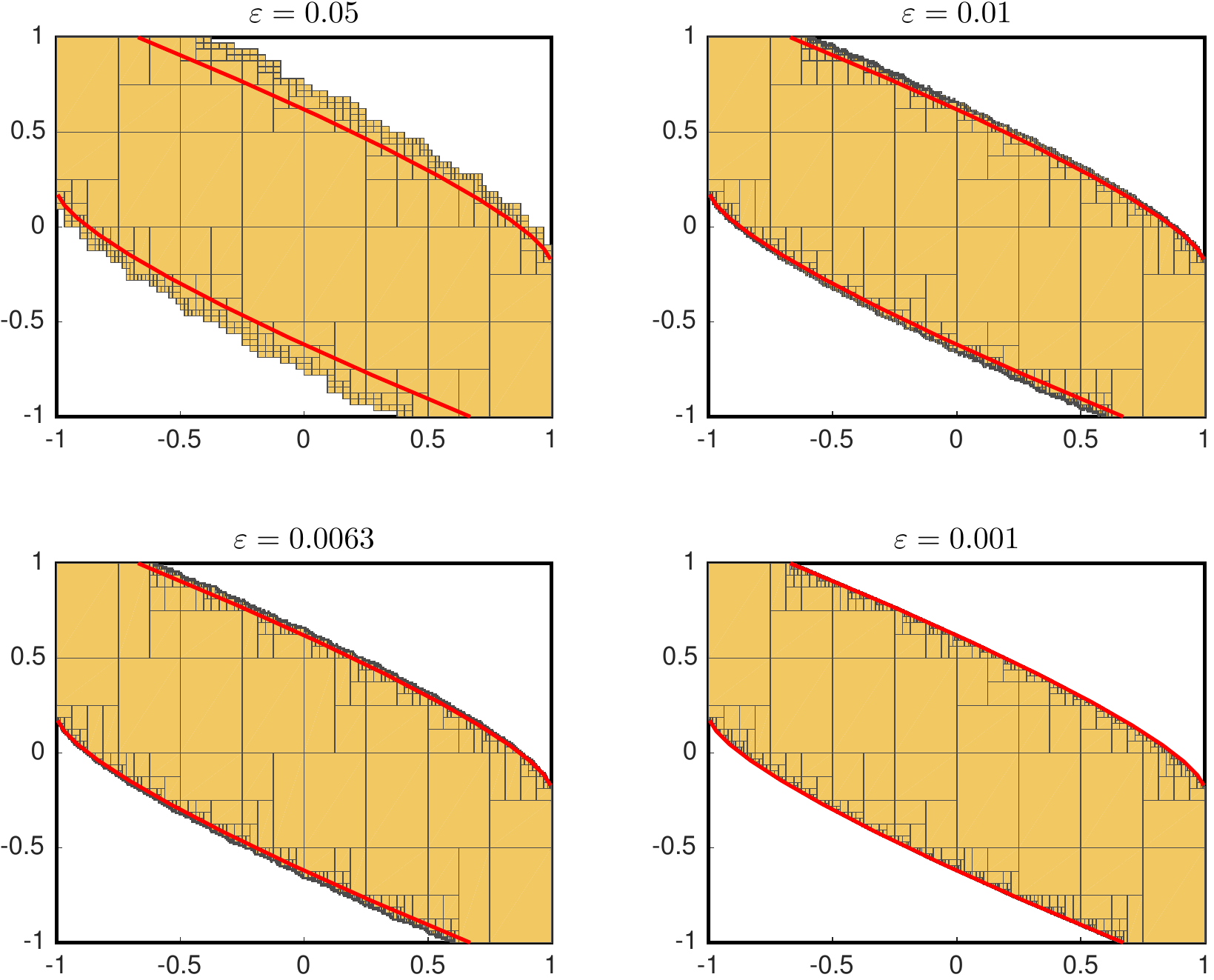}
  \caption{Outer approximations of $\I(\Omega)$ with different precision parameters.}
  \label{fig:LTIouters}
\end{figure}

\end{example}

\subsection{Controlled invariant sets under robust condition}
Outer approximations of the maximal invariant sets are not invariant by definition. However, we are able to find inner-approximations that are invariant, provided that the system satisfies a certain robustly invariant condition, which is introduced in the following definition.

\begin{definition} \label{def:rbmargin}\em
A compact set $\Omega$ is said to be a \emph{$r$-robustly controlled invariant set} for system (\ref{eq:sw}) 
if
  \begin{equation}
    \label{eq:invmargin}
    \Omega \subset \pre(\Omega\ominus \mathcal{B}_r),
  \end{equation}
where $r\geq 0$, and $\ominus$ is the Pontryagin difference. Denote by $r^*$ the supremum of $r$ such that (\ref{eq:invmargin}) is satisfied, which is called the \emph{robust invariance margin} of $\Omega$.
\end{definition}

Clearly, if a set is $r$-robustly controlled invariant, then it is also $r'$-robustly controlled invariant for all $r'\in (0,r)$. Systems with positive robust invariance margins of a given set $\Omega$ are able to be controlled invariant even under a certain degree of uncertainties, including exogenous disturbances, modeling errors and computational errors that are introduced in computing reachable sets. The larger the margin is, the higher degree of uncertainties the system can tolerate.

\begin{proposition}\em \label{prop:rb0}
  Suppose that Assumption \ref{asp:f} holds. Given a compact set $\Omega\subset \Real^n$, let $\I(\Omega)$ be the maximal invariant set within $\Omega$. Suppose that $\I(\Omega)\neq \Omega$. Then $\I(\Omega)$ is with zero robust invariance margin.
\end{proposition}
\begin{proof}
  We prove this by showing that some boundary points of $\I(\Omega)$ will be mapped into the boundary of $\I(\Omega)$, which implies that its robust invariance margin is zero. For the purpose of contradiction, we assume $x\in (\partial \I(\Omega) \cap \text{int}(\Omega))$, and there exists a $p \in \M$ that $f_p(x) \in \text{int}(\Omega)$. That implies there exists a $r>0$ such that $B_r(f_p(x))\subset \Omega$. By continuity of $f_p$, we can find a $\delta(r)>0$ such that any $x' \in B_{\delta(r)}(x)$ satisfies $f_p(x')\in B_r(f_p(x))$, and thus $f_p(B_{\delta(r)}(x))\subset \Omega$, which means $x$ is an interior point of $\Omega$. This is contradictory with the condition.
\end{proof}

Likewise, we are interested in finding the maximal $r$-robustly invariant sets within a given compact set. Similar to Definition \ref{def:maxinvc}, we can have the following definition.

\begin{definition}\em
  Given a set $\Omega\subset \Real^n$, the set $\I_r(\Omega)$ is said to be the maximal $r$-robustly invariant set inside $\Omega$ for system (\ref{eq:sw}), if it is $r$-robustly invariant and contains every $r$-robustly invariant sets inside $\Omega$.
\end{definition}

We modify the mapping (\ref{eq:mapi}) into
\begin{equation}
  \label{eq:mapir}
  I_r(X)=\pre(X\ominus \mathcal{B}_r)\cap X.
\end{equation}
Similarly, we denote by $I_r^j$ ($j\in \Z_{>0}$) the $j$-times repeated compositions of the mapping $I_r$. The conceptual procedure of computing a $r$-robustly controlled invariant set within a given compact set $\Omega$ can be obtained by modifying Algorithm \ref{alg:mcis0}.
\begin{algorithm}[H]
  \centering
  \caption{Computation of $\I_r(\Omega)$}
  \label{alg:mcis1}
  \begin{algorithmic}[1]
    \Require $\Omega,\pre, r$
    \State $\widetilde{X}=\Omega, X= \varnothing$
    \While{$X\neq \widetilde{X}$}
    \State $X=\widetilde{X}, \widetilde{X}=\pre(X\ominus \mathcal{B}_{r}) \cap X$
    \EndWhile
    \State \Return $X$
  \end{algorithmic}
\end{algorithm}

\begin{proposition}\em \label{prop:convrb}
  Let $\Omega\subset \Real^n$ be closed. Then $\I_r(\Omega)=\lim_{j\to \infty}I_r^j(\Omega)$.
\end{proposition}
\begin{proof}
  First of all, from (\ref{eq:mapi}), it is easy to see that $I^{j-1}\subset I^j$ for all $j\geq 1$. According to \cite[Theorem 2.1]{KolmanovskyG98}, $\Omega\ominus \mathcal{B}_r$ is closed. By Proposition \ref{prop:pre} \ref{itm:pre1}, the set $\pre(\Omega\ominus \mathcal{B}_r)$, and hence $I_r^j(\Omega)$, is closed. With the fact that $\{I_r^{j}\}$ is a non-increasing sequence, the set limit of $I^j(\Omega)$ exists and equals the closed set $\bigcap_{j=1}^\infty I_r^j(\Omega)$. If $I_r^j(\Omega)$ is nonempty for every $j>0$, then $\bigcap_{j=1}^\infty I_r^j(\Omega)$ is nonempty \cite[p. 225, 1.6]{DugundjiBook67}.

Second, we show that $\lim_{j\to \infty}I_r^j(\Omega)\subset \I_r(\Omega)$. If $\lim_{j\to\infty}I_r^j(\Omega)$ is empty, this trivially holds. If not, we have to show that $\bigcap_{j=1}^\infty I_r^j(\Omega)$ is $r$-robustly controlled invariant within $\Omega$. For any $x\in \lim_{j\ra\infty}I_r^j(\Omega)$, $x\in I_r^j(\Omega)$ for all $j\ge 1$. It follows that there exists $p_j\in \M$ such that $f_{p_j}(x)\in (I_r^{j-1}(\Omega)\ominus\mathcal{B}_r)$ for all $j\ge 1$. Since $\M$ is finite, the sequence $\set{p_j}_{j=1}^\infty\subset \M$ must admit a constant subsequence. In other words, there exists $p\in\M$ such that $f_p(x)\in I_r^{j-1}(\Omega)$ for infinitely many $j\ge 1$. By monotonicity of the sequence $I_r^j(\Omega)$, this implies $f_p(x)\in \bigcap_{j=1}^{\infty}I_r^j(\Omega)=\lim_{j\ra\infty}I_r^j(\Omega)$. Hence $\lim_{j\ra\infty}I_r^j(\Omega)$ is a controlled invariant subset of $\Omega$.

Third, we use mathematical induction to prove $\I_r(\Omega)\subset \lim_{j\to \infty}I_r^j(\Omega)$, i.e., $\lim_{j\to \infty}I_r^j(\Omega)$ is maximal. We assume $\I_r(\Omega)\ominus\mathcal{B}_r\neq \varnothing$, otherwise $\I_r(\Omega)=\varnothing$, which means the conclusion trivially holds. For the base case $j=0$, we have $\I_r(\Omega)\subset I_r^0(\Omega)=\Omega$. Suppose $\I_r(\Omega)\subset I_r^j(\Omega)$. By (\ref{eq:invmargin}), for any $x\in (I_r^j(\Omega)\setminus I_r^{j+1}(\Omega))$, $f_p(x)\notin (I_r^{j}(\Omega)\ominus\mathcal{B}_r)$ for all $p\in \M$, which also means $f_p(x)\notin (\I_r(\Omega)\ominus\mathcal{B}_r)$. By definition of $\I_r(\Omega)$, $x\notin \I_r(\Omega)$. It follows that $\I_r(\Omega)\subset I_r^{j+1}(\Omega)$. Hence, $\I_r(\Omega)\subset \bigcap_{j=1}^\infty I_r^j(\Omega)= \lim_{j\to \infty}I_r^j(\Omega)$. This completes the proof.
\end{proof}

Based on Algorithm \ref{alg:cpre} for computing pre-images, we have the following algorithm for finding controlled invariant sets.

\begin{algorithm}[htbp]
  \caption{Inner Approximation of $\I(\Omega)$}
  \label{alg:innerc}
  \begin{algorithmic}[1]
  \Require $\set{[f_p]}_{p\in\P},\Omega, \varepsilon$
    \State $X\leftarrow [\Omega],\widetilde{Y}\leftarrow\Omega, Y\leftarrow\varnothing$

    \While{$Y \neq \widetilde{Y}$}
    
    \State $Y=\widetilde{Y}$
    \State $[\underline{X}, \Delta X, X_c,\C]=\text{CPre}(\set{[f_p]}_{p\in\P},Y,X,\varepsilon)$
    \State $\widetilde{Y}\leftarrow \cup_{[x]\in\underline{X}}[x]$
    \State $X\leftarrow \underline{X}$
    
    \EndWhile	
    \State  \Return $Y,\C$
   \end{algorithmic}
\end{algorithm}

\begin{theorem}\em \label{prop:innerthm}
  Let $\Omega\subset \Real^n$ be a compact set and Assumption \ref{asp:f} holds. Denote by $\underline{Y}^\varepsilon$ the output of Algorithm \ref{alg:innerc} for a given precision $\varepsilon$. Then Algorithm \ref{alg:innerc} terminates in a finite number of steps. Furthermore, if $\rho_1\varepsilon\leq r$, then the following conclusions hold:
\begin{enumerate}[label=(\roman*)]
\item \label{itm:inner1} If $\underline{Y}^\varepsilon=\varnothing$, then system (\ref{eq:sw}) does not have a $r$-robustly controlled invariant set contained in $\Omega$;
\item \label{itm:inner2} if $\underline{Y}^\varepsilon\neq\varnothing$, then $\underline{Y}^\varepsilon$ is controlled invariant, and
\end{enumerate}
  \begin{equation}
    \label{eq:cre}
    \I_r(\Omega)\subset \underline{Y}^\varepsilon \subset \I(\Omega).
  \end{equation}
\end{theorem}

\begin{proof}
  Again, we show that Algorithm \ref{alg:innerc} terminates in finite steps first. Similar to the argument for Algorithm \ref{alg:outerc}, we denote by $\{Y_j\}$ ($j\geq 0$) the resulting non-increasing sequence with $Y_0=\Omega$, and $Y_{j+1}=Y_j\setminus (\Delta X_j\cup X_{c,j})$. If $Y_{j+1}\neq Y_j$, then $(\Delta X_j\cup X_{c,j})\neq \varnothing$, and $\{Y_j\}$ is a strictly decreasing sequence. Under a given precision $\varepsilon$, for all $j\geq 0$, $Y_j$ is represented by a finite number of intervals, as a result of the compactness of $\Omega$. Then there must exists an $N\in \Z_{>0}$ such that $Y_N=\varnothing$, which results in $Y_{N+1}=\varnothing$. This means that at the beginning of the $(N+2)$ th iteration, $Y=\widetilde{Y}$ is satisfied. Hence, this algorithm will terminate in finite iterations.

To prove \ref{itm:inner1} and \ref{itm:inner2}, suppose that $\{X_j\}$ ($j\geq 0$) is the sequence generated by Algorithm \ref{alg:mcis1}. According to the mapping (\ref{eq:mapir}), we have
\begin{equation*}
  \begin{aligned}
    X_0&=\Omega,\\
    X_{j+1}&=(\pre(X_{j}\ominus\mathcal{B}_r)\cap X_{j}).
  \end{aligned}
\end{equation*}
By Theorem \ref{prop:cpre},
\begin{equation*}
  \begin{aligned}
    Y_0&=\Omega,\\
    Y_{j+1}&\supset (\pre(Y_j\ominus\mathcal{B}_{\rho_1 \varepsilon})\cap Y_j),\\
    Y_{j+1}&\subset (\pre(Y_j)\cap Y_j).
  \end{aligned}
\end{equation*}
It is clear that both $\{X_j\}$ and $\{Y_j\}$ are non-increasing sequences. Considering that $X_0=Y_0=\Omega$, $X_1=(\pre(X_0\ominus\mathcal{B}_r)\cap X_0)$, $(\pre(Y_0\ominus\mathcal{B}_{\rho_1 \varepsilon}) \cap Y_0)\subset Y_1\subset\pre(Y_0)$, we have $X_1\subset Y_1$. Suppose that $X_j\subset Y_j$ for all $j\geq 1$. Since $\rho_1\varepsilon\leq r$, we have $X_{j+1}\subset Y_{j+1}$. This means that for all $j\geq 0$, $X_j\subset Y_j$.

If $\underline{Y}^\varepsilon=\varnothing$, then there exists some integer $N>0$ such that $Y_N=\varnothing$. It follows that $X_N=\varnothing$, and $\I_r(\Omega)=\bigcap_{j=1}^\infty I_r(\Omega)=\varnothing$. Hence, \ref{itm:inner1} is proved.

If $\underline{Y}^\varepsilon\neq\varnothing$, then there exists an integer $J>0$ such that $\underline{Y}^\varepsilon=Y_J=Y_{J+1}\supset X_J\supset (\bigcap_{j=1}^\infty I_r(\Omega))=\I_r(\Omega)$. Since $Y_J=Y_{J+1}\subset (\pre(Y_J)\cap Y_{J})$, we have $Y_J\subset \pre(Y_J)$, i.e., $Y_J$ is controlled invariant by definition. Hence $Y_J\subset \I(\Omega)$. That completes the proof.
\end{proof}

\begin{remark}\em
  Theorem \ref{prop:innerthm} serves as a rule for choosing the precision parameter $\varepsilon$ if the robust invariance margin is known in prior. For switched linear systems, one can refer to quadratic Lyapunov functions for estimations of their robust invariance margins. For switched nonlinear systems, however, it is nontrivial to obtain such margins either analytically or numerically because of the generality of the dynamics. In practice, it is not necessary to know this margin before computation. As implied in Theorem \ref{prop:innerthm}, one can choose a sufficiently small $\varepsilon$ while satisfying the computational performance requirement. If the number of intervals is a major concern, a better way is to start with a larger $\varepsilon$ and iteratively reduce it until Algorithm \ref{alg:innerc} returns a nonempty result. This way of iteratively decreasing $\varepsilon$ and recomputing controlled invariant sets can also be considered as a numerical method for evaluating robust invariance margins.

Theorem \ref{prop:innerthm} is applicable to general nonlinear dynamics. A special convergence result involving compact and convex sets containing the origin for linear systems can be found in \cite[Theorem 3.1]{Blanchini94}. 
\end{remark}

In the two examples given below, we illustrate how Algorithm \ref{alg:innerc} and Theorem \ref{prop:innerthm} are applied in controlled invariant sets computation. The second example shows the performance of our method applied to nonlinear systems.

\begin{example}\em \label{eg:lti2}
  Consider again the LTI system in Example \ref{eg:lti}. The goal is to compute an invariant inner approximation of the maximal positively invariant set within the same compact set $\Omega$. 

To estimate the robust invariance margin, we use an invariant set candidate $\{ x\in\Real^2 \sv x^TPx\leq \gamma^2, \gamma >0 \}$, which is the level set of the quadratic Lyapunov function $V(x)=x^TPx$. The symmetric positive definite matrix $P$ is determined by the discrete Lyapunov equation $A^TPA-P+Q=0$, where the matrix $Q$ is symmetric and positive definite, and $\gamma$ is chosen such that the candidate is inner tangent to $\Omega$. Setting $Q=\text{diag}(0.1,0.1)$, we get the robust invariance margin $r=0.0031$. Taking $\rho_1=1.196$, which is the Euclidean norm of $A$, the precision parameter $\varepsilon$ is set to $0.001$ according to Theorem \ref{prop:innerthm}. The resultant inner approximation is shown in Figure \ref{fig:LTIinner}.
\begin{figure}[htbp]
  \centering
  \includegraphics[scale=0.65]{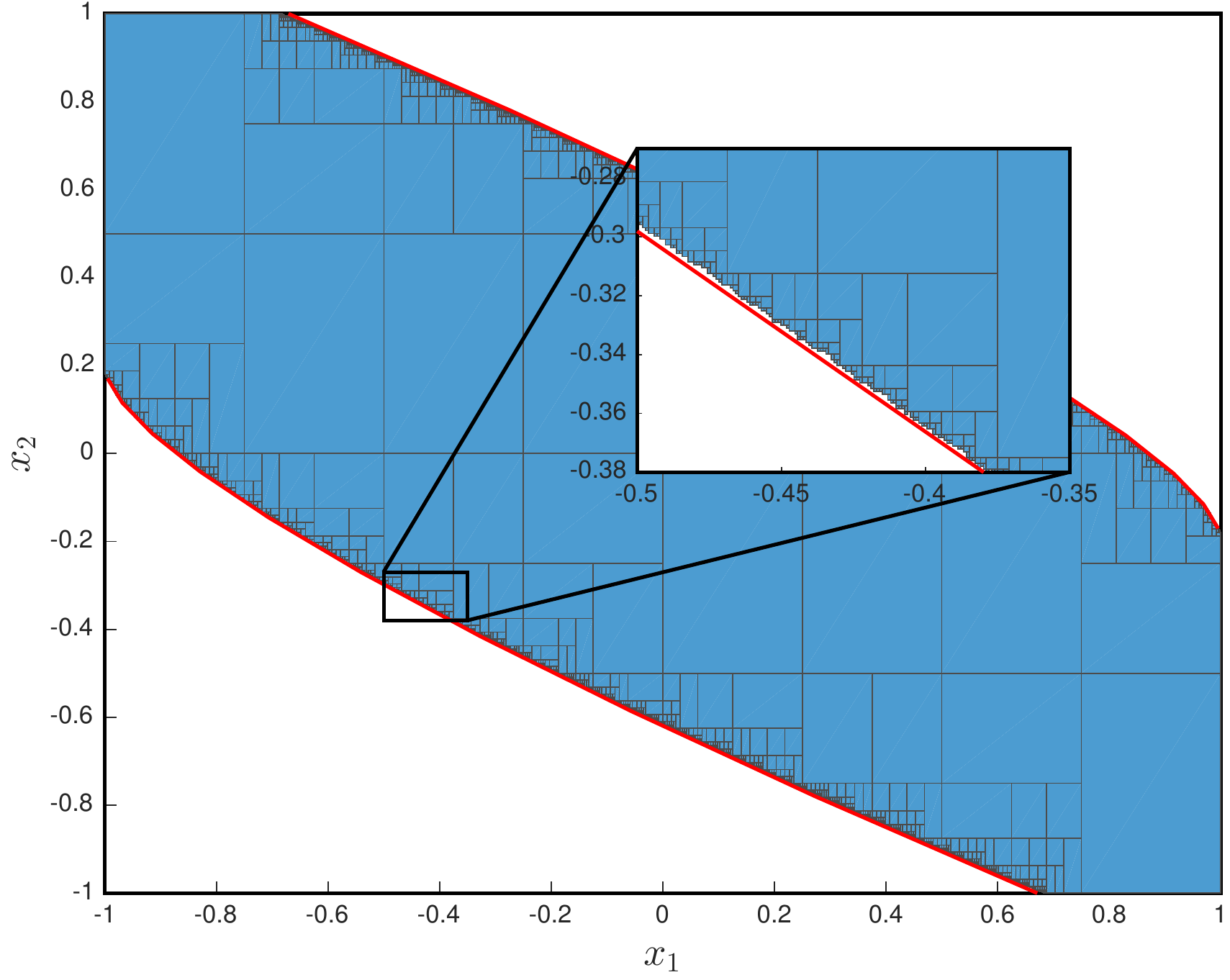}
  \caption{The inner approximation of $\I(\Omega)$ for Example \ref{eg:lti} with $\varepsilon=0.001$, which is marked by the union of the blue intervals.}
  \label{fig:LTIinner}
\end{figure}

The estimation of the robust invariance margin is conservative, and in practice, the precision parameter can be set larger.

\end{example}

\begin{example}\em \label{eg:8p6}
  Consider a discrete-time version of a second-order nonlinear system taken from \cite[Example 8.6]{KhalilBook02} as follows:
\begin{equation*}
  \begin{aligned}
    x_1(k+1)&=x_1(k)+0.1x_2(k)\\
    x_2(k+1)&=-0.1x_1(k)+0.033x_1^3(k)+0.9x_2(k).
  \end{aligned}
\end{equation*}

It has three isolated equilibrium points at $(0,0)$, $(\sqrt{3},0)$ and $(-\sqrt{3},0)$. The region between the manifolds that pass through $(\sqrt{3},0)$ and $(-\sqrt{3},0)$ is the maximal positively invariant set, which is difficult to express analytically. Figure \ref{fig:8p6} displays the outer and inner approximations.
\begin{figure}[htbp]
  \centering
  \begin{subfigure}[c]{0.45\textwidth}
    \centering
    \includegraphics[width=\linewidth]{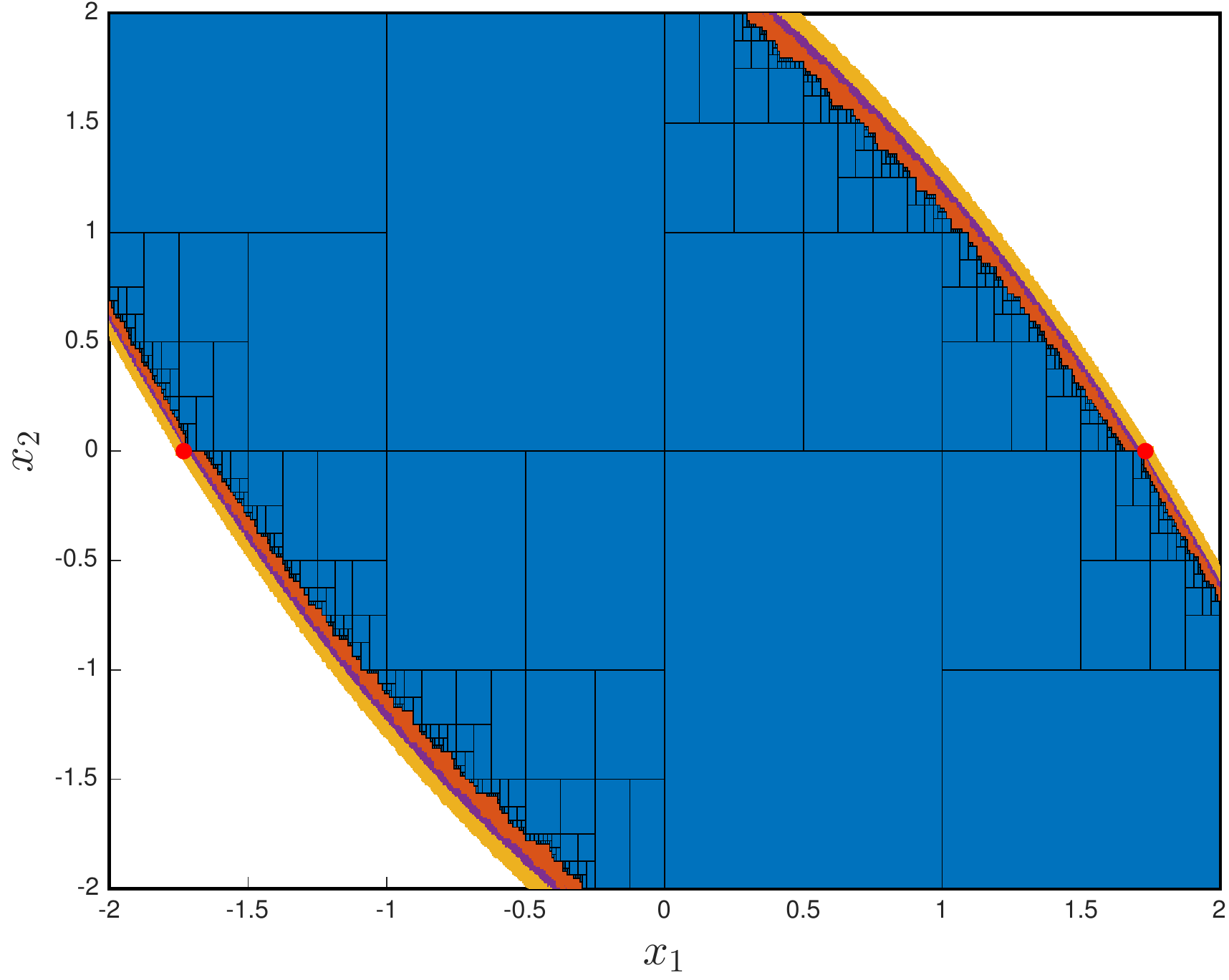}
  \end{subfigure}
  \begin{subfigure}[c]{0.35\textwidth}
    \centering
    \vspace{0pt}
    \begin{tabular}{c|c}
      \hline
      $\varepsilon$ & AI/AO \\
            \hline
      0.004 & 0.979 \\
      \hline
      0.01 &  0.969 \\
      \hline
      0.03 &  0.949 \\
      \hline
    \end{tabular}
  \end{subfigure}

  \caption{Approximations of $\I(\Omega)$ for example \ref{eg:8p6}. Regions from outside to inside: outer approximation with $\varepsilon=0.004$, inner approximations with $\varepsilon=0.004,0.01,0.03$, respectively. Right table: ``AI/AO'' denotes the ratio between the areas of an inner approximation and the outer approximation area. }
  \label{fig:8p6}
\end{figure}

The result shows that the inner approximation obtained using a smaller precision parameter approximates the exact maximal positively invariant set more precisely.

\end{example}

The robustly controlled invariance condition is critical for the successful application of Algorithm \ref{alg:innerc}. Consider a discrete-time system $x(t+1)=A_\theta x(t)$, where
\begin{equation*}
  A_\theta=\begin{bmatrix} \cos{\theta} & -\sin{\theta}\\ \sin{\theta} & \cos{\theta} \end{bmatrix}.
\end{equation*}
Under this dynamics, every state $x\in \Real^2$ moves on a circle centered at the origin. It is evident that the robust invariance margin is 0, and Algorithm \ref{alg:innerc} will return an empty set.

\subsection{Complexity analysis}
For both Algorithm \ref{alg:outerc} and \ref{alg:innerc}, the number of iterations varies for different systems, depending on the invariance property of the dynamics on a given area. Hence, we only analyze the complexity of Algorithm \ref{alg:cpre} in this section. Denote by $|\Omega|$ the number of the intervals that represents $\Omega$ and $|\M|$ the number of switching modes. Computing the interval inclusion function $[f_p]$ takes a constant time, and checking intersection has an average complexity of $\mathcal{O}(\log(|\Omega|))$ if a search-based algorithm is used. In the best case, i.e., no subdivisions are conducted, the overall complexity of Algorithm \ref{alg:cpre} is $\mathcal{O}(|\M||\Omega|\log(|\Omega|))$. In the worst case, the set $\Omega$ is subdivided into $(w([\Omega])/\varepsilon+1)^n$ intervals \cite{JaulinBook01}. Let $|\Omega|_{\text{max}}=(w([\Omega])/\varepsilon+1)^n$. Then the worst-case complexity is $\mathcal{O}(|\M||\Omega|_{\text{max}}\log(|\Omega|_{\text{max}}))$.

\section{Extraction of invariance controller }
This section is dedicated to Problem \ref{pb:invctl} \ref{itm:pb2}, i.e., the construction of an invariance controller using the information recorded by Algorithm \ref{alg:innerc}.

For the sake of simplicity, we denote by $f_p(X)$ the image of $X\subset \Real^n$ under the dynamics $f_p$ for all $p\in \M$. Moreover, Let $\Y=\{Y_1,Y_2,\cdots, Y_N\}$ and $\C=\{C_1,C_2, \cdots, C_N\}$ be the outputs of Algorithm \ref{alg:innerc}, where $N\in \Z_{>0}$ denotes the number of intervals and 
$C_i$ contains a set of modes such that $\forall p\in C_i$, $f_{p}(Y_i)\subset (\bigcup_{i\in N}Y_i)$.

To explicitly extract an invariance controller, we introduce the following definition.

\begin{definition} \label{def:part} \em
Given a set $\Omega \subset\Real^n$, a finite collection of sets 
$$\P=\{P_1,P_2,\cdots,P_N\},$$ 
is said to be a \emph{partition} of $\Omega$, if the following conditions are satisfied:
  \begin{enumerate}
  \item $P_i \subset \Omega$, for all $i\in\set{1,\cdots,N}$;
  \item $\mathrm{int}(P_i) \cap \mathrm{int}(P_j)=\varnothing$, for all  $i,j\in\set{1,\cdots,N}$;
  \item $\Omega \subset \bigcup_{i=1}^N P_i$.
  \end{enumerate}
Each element $P_i$ of the partition $\P$ is called a \emph{cell}.
\end{definition}

If the dynamics of each subsystem of (\ref{eq:sw}) is Lipschitz continuous on $\Omega$, we can show, in the following theorem, the existence of a partition-based invariance controller on $\Omega$.

\begin{theorem}\em \label{thm:exist}
  Let $\Omega\subset \Real^n$ be compact. Suppose that Assumption \ref{asp:f} holds on $\Omega$. If $\Omega$ is a $r$-robustly ($r>0$) controlled invariant set, then there exists a partition $\P=\{P_1,P_2,\cdots,P_N\}$ of $\Omega$ and an invariance controller $c: \Real^n\to 2^\M$ with
\begin{equation}
  \label{eq:cmap}
  c(x)=\bigcup_{i\in N}\psi_{P_i}(x), \quad x\in \Omega.
\end{equation}
The map $\psi_{P_i}$ is given by
\begin{equation}
    \label{eq:cell}
    \psi_{P_i}(x)=
    \begin{cases}
      \varnothing & \quad \text{if } x\notin P_i,\\
      \{p_{i_k}\}  & \quad \text{if } x\in P_i,\\
    \end{cases}
  \end{equation}
where $p_{i_k}\in \M$ for $i\in\set{1,\cdots,N}$, $k\in \{1,\cdots, |\M|\}$.
\end{theorem}

\begin{proof}
  We prove it by constructing a partition of $\Omega$ and the corresponding invariance controller.

Since $\Omega$ is a $r$-robustly ($r>0$) controlled invariant set, it satisfies (\ref{eq:invmargin}). Consider an arbitrary point $x\in\Omega$, there exists a switching mode $p\in \M$ such that $f_p(x)\in (\Omega\ominus \mathcal{B}_r)$. Under Assumption \ref{asp:f}, any state $x'\in \mathcal{B}_\delta(x)$ ($\delta>0$) satisfies
\begin{equation*}
  \|f_p(x')-f_p(x)\| \leq \rho_1\|x'-x\|\leq \rho_1\delta,
\end{equation*}
for all $p\in\M$.

Let $\delta< r/\rho_1$. Then $\|f_p(x')-f_p(x)\| <r$ for all $x'\in \mathcal{B}_\delta(x)$. It follows that any point in $\mathcal{B}_{\delta}(x)$ can be controlled inside $\Omega$.

Given that $\Omega$ is a compact set, by Borel-Lebesgue covering theorem, there exists a finite subcovering $\set{\Omega_i}$, $i=1,2,\cdots, N$, of any open covering of $\Omega$ such that it covers $\Omega$. Then we can select a finite set of points $\{x_i\}_{i=1}^N \in \Omega$ such that 
\begin{equation*}
  \Omega \subset \left (\bigcup_{i=1}^N \mathcal{B}_{\delta}(x_i) \right ).
\end{equation*}

Now, we construct a partition as follows:
\begin{equation*}
  \begin{aligned}
    P_1&=\mathcal{B}_{\delta}(x_1),\\
    P_2&=\mathcal{B}_{\delta}(x_2) \setminus \mathcal{B}_{\delta}(x_1),\\
    \cdots,\\
    P_N&=\mathcal{B}_{\delta}(x_N) \setminus \left(\bigcup_{i=1}^{N-1} \mathcal{B}_{\delta}(x_i)\right).
  \end{aligned}
\end{equation*}

Assume that $p_i\in \M$ renders $f_{p_i}(x_i)\in (\Omega\ominus \mathcal{B}_r)$. Using the map defined as
\begin{equation*}
    \psi_{P_i}= 
    \begin{cases}
      \varnothing & \quad \text{if } x\notin P_i,\\
      p_i  & \quad \text{if } x\in P_i,\\
    \end{cases}
  \end{equation*}
any state of the resulting sequence $\set{x_k}_{k=0}^\infty$ will be controlled into one of the cells of the partition $\P$, which proves the statement.
\end{proof}

An advantage of Algorithm \ref{alg:innerc} is that the output set $\Y$ naturally defines a partition of $\Omega$. This partition, if not empty, covers the one-step forward reachable set of each cell of $\Y$, guaranteeing that the invariance control problem is feasible.

\begin{proposition} \em \label{prop:invctlr}
  Assume that Algorithm \ref{alg:innerc} terminates with a nonempty set $\Y$. Then the invariance controller $c$ in the form of (\ref{eq:cmap}) with $\P=\Y$ and
  \begin{equation}
    \label{eq:cellinv}
    \psi_{Y_i}= 
    \begin{cases}
      \varnothing & \quad \text{if } x\notin Y_i,\\
      C_i & \quad \text{if } x\in Y_i,\\
    \end{cases}
  \end{equation}
renders the closed-loop system invariant with respect to $\Omega$.
\end{proposition}

\begin{proof}
  The switching controller composed of (\ref{eq:cellinv}) 
conforms to the one constructed in the proof of Theorem \ref{thm:exist}. Therefore, it renders the controlled system invariant with respect to $\Omega$.
\end{proof}

\begin{remark}\em
  There is a close connection between our method and abstraction-based approaches using transition systems. A transition system is a tuple $\T=(\Q,\Q_0,\A,\rightarrow_{\T})$, where $\Q$ ($Q_0$) is the set of (initial) states, $\A$ is the set of actions, and $\rightarrow_{\T} \, \subset \Q \times \A \times \Q$ is the transition relation. The closed-loop system with an invariance controller given by (\ref{eq:cmap}) and (\ref{eq:cellinv}) can be formulated as a controlled transition system. Suppose that $\Y$ and $\C$ are the corresponding sets obtained from Algorithm \ref{alg:innerc}. Then the abstraction can be constructed as
$\T_{\Y,\C}=(\Y,\Y,\M,\rightarrow_{\C})$, where $(Y_i,\,p,\,Y_j) \in \rightarrow_{{\C}}$ if and only if there is a switching mode $p\in C_i$ such that $Y_j\cap f_p(Y_i) \neq \varnothing$.
\end{remark}

\section{Examples}
\label{sec:exps}

In this section, we present three case studies. In each of these examples, subsystems do not have common equilibrium points. Hence, Lyapunov-based methods 
can not be applied. We compare the performance of our method with abstraction-based methods in terms of computational time and abstraction size.

\subsection{Boost DC-DC converter}
   Consider a typical boost DC-DC converter \cite{GirardPT10} with two switching modes and linear affine dynamics $\dot{x}=A_px+b$, where $p=1,2$ and
  \begin{equation*}
    \begin{aligned}
    A_1&=\begin{bmatrix} -\frac{r_l}{x_l} & 0 \\  0 & -\frac{1}{x_c(r_c+r_0)} \end{bmatrix},\\
    A_2&=\begin{bmatrix} -\frac{1}{x_l}(r_l + \frac{r_0r_c}{r_0+r_c}) & -\frac{r_0}{x_l(r_0+r_c)}\\
    \frac{r_0}{x_c(r_0+r_c)} & -\frac{1}{x_c(r_0+r_c)} \end{bmatrix},\\
  b&=\begin{bmatrix} \frac{v_s}{v_l} & 0 \end{bmatrix},
\end{aligned}
  \end{equation*}
where $x_c=70$ per unit (p.u.) and $r_c=0.005$p.u. denote the capacity and resistance of the capacitor; $x_l=3$p.u. and $r_l=0.05$p.u. denote the inductance and resistance of the indcutor; the load resistance and the source voltage are given by $ r_0=1$p.u. and $v_s=1$p.u., respectively. With sampling time $\tau_s=0.5$s, we use the discrete-time model $x(t+1)=e^{A_p\tau_s}x(t)+\int_0^{\tau_s} e^{\tau_s-s}b\,\mathrm{d}s$.

The invariance specification is given by $\Omega=[1.15,1.55]\times[1.09,1.17]$
. Running Algorithm \ref{alg:innerc} with $\varepsilon=0.001$, we obtain the maximal controlled invariant set represented by intervals and the corresponding feasible control inputs for each interval. The invariance controller constructed by Proposition \ref{prop:invctlr} serves as a least restrictive controller for the boost DC-DC converter. We apply the control policy that keeps the switching mode unchanged unless the state is going to leave $\Omega$. The state evolution of the closed-loop system with initial condition $x_0=[1.2,1.12]$
, shown in Figure \ref{fig:dcdc}, is confined to the controlled invariant set (intervals of gray color in (a)) of $\Omega$ as required.
\begin{figure}[htbp]
  \centering
  \begin{subfigure}[b]{0.5\textwidth}
    \centering
    \includegraphics[scale=0.65]{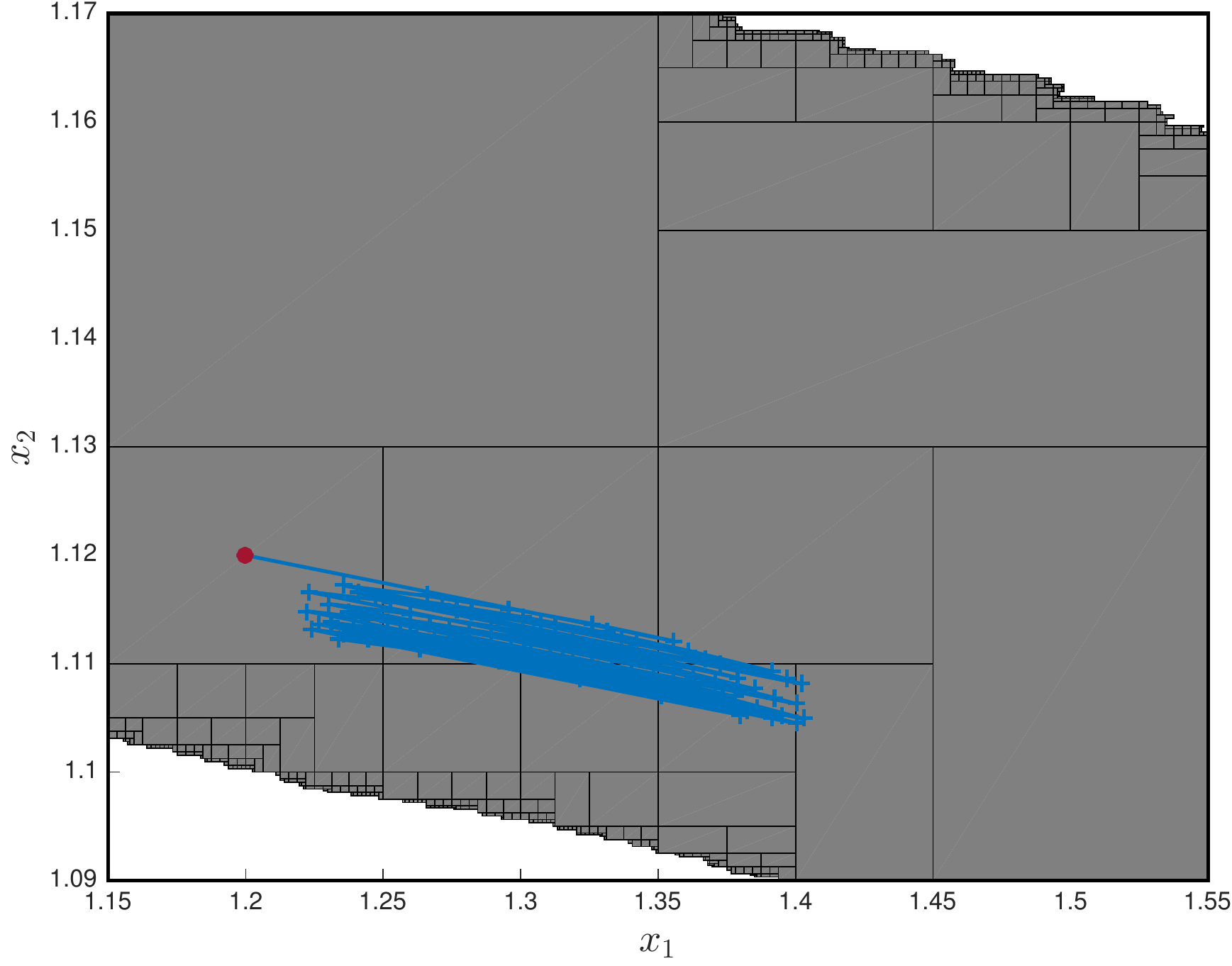}
    \caption{Closed-loop phase portrait.}
  \end{subfigure}
  \begin{subfigure}[b]{0.5\textwidth}
    \centering
    \includegraphics[scale=0.65]{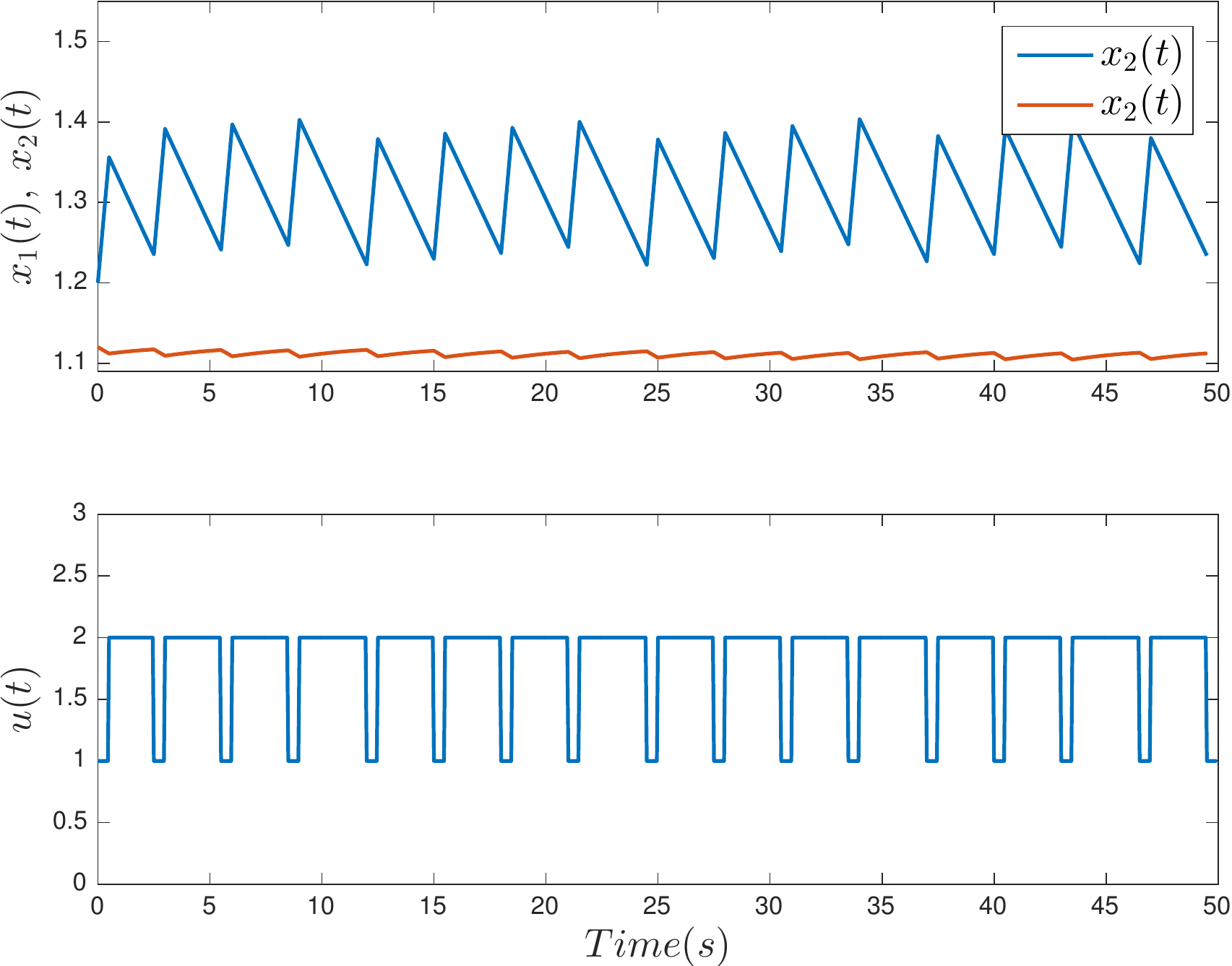}
    \caption{Time history of closed-loop states and control variables.}
  \end{subfigure}
  \caption{Invariance control results for boost DC-DC converter.}
  \label{fig:dcdc}
\end{figure}

Implemented in Matlab, our algorithm returns a nonempty invariant set and the corresponding invariance controller. We compare the run time of our algorithm with those of abstraction-based methods reported in \cite{RunggerZ16} in Table \ref{tbl:runtimes}, where ``$t_{abs}$'' stands for the time spent on computing abstractions, and ``$t_{syn}$'' is the time spent on control synthesis. In terms of efficiency, our algorithm outperforms other existing methods, except CoSyMA, which uses hash tables for controller synthesis and is more efficient than search-based methods.

\begin{table}[htbp]
\small \caption{Comparison of run times}
\label{tbl:runtimes}
\centering
\begin{tabular}{c||c|c|c}
\hline
 & CPU [GHz] & $t_{\text{abs}}$[s] & $t_{\text{syn}}$[s] \\
\hline
Pessoa \cite{MazoDT10} & i7 3.5 & 478.7 & 65.2 \\
\hline
SCOTS \cite{RunggerZ16} & i7 3.5 & 18.1 & 75.4 \\
\hline
CoSyMA \cite{MouelhiGG13} & N/A & N/A & 8.32\\
\hline
intvl  & i5 2.4 & 0 & 43.9%6.77
 \\
\hline
\end{tabular}
\end{table}

\subsection{Inverted pendulum on cart}
In this example, we aim to control an inverted pendulum on a cart. If the position of the cart need not to be controlled, the model can be simplified to
\begin{equation*}
  \begin{aligned}
    \dot{x}_1&=x_2,\\
    \dot{x}_2&=\frac{mgl}{J_t}\sin{x_1} -\frac{b}{J_t}x_2 +\frac{l}{J_t}\cos{x_1}u,
  \end{aligned}
\end{equation*}
where $J_t=J+ml^2$, $m$ the mass of the cart, $g$ is the acceleration of gravity, $l$ is the length of the pendulum, $J$ is the moment of inertia of the pendulum, $b$ is the coefficient of friction for cart, and $u$ is the force applied to the cart. The state $x_1$ denotes the angle of the pendulum to the upper vertical line $\varphi$ (rad), and $x_2$ is $\dot{\varphi}$ (rad/s). In our simulation, the parameters are taken as $m=0.2$kg, $g=9.8\text{m}/\text{s}^2$, $l=0.3$m, $J=0.006$kg$\text{m}^2$, $b=0.1\text{N}/\text{m}/\text{s}$.

Applying Euler's method with a sampling time $\tau_s=0.01$s, we obtain an approximated discrete-time model for the inverted pendulum. Assuming that the control variable $u$ takes values from a finite set $U$, this example can be seen as a discrete-time switched system, which is neither globally asymptotically stable nor incrementally asymptotically stable on $\Omega$. While approximate bisimilar models \cite{GirardPT10} are not applicable, a similar model (or non-deterministic abstraction) can still be constructed by a proper growth bound $\beta:\Real^n\times U \to \Real^n_{\geq 0}$ (see \cite{ZamaniPMT12} for the definition) with
\begin{equation*}
  \|f_p(x)-f_p(y)\|\leq \beta(\|x-y\|,u), \forall x,y\in \Real^n, \forall u\in U,
\end{equation*}
for estimating the distance between two trajectories.

For comparison of our method with abstraction-based methods, 
we choose a local growth bound $\beta(\eta,u)$:
\begin{equation*}
  \begin{bmatrix}
    0.5\eta_1+0.005\eta_2\\ (0.49\cos{\frac{\eta_1}{2}}+0.25|u|\sin{\frac{\eta_1}{2}})|\sin{(x_1+\frac{\eta_1}{4})}|+0.48\eta_2
  \end{bmatrix},
\end{equation*}
where $\eta=[\eta_1,\eta_2]$ is the grid width, $x_1$ is the center of the current grid.

Two invariance control specifications $\Omega_1=[-0.05,0.05]\times[-0.01,0.01]$ and $\Omega_2=[0.10, 0.17]\times[-0.01,0.01]$ 
are considered. The computational settings and results are summarized in Tables \ref{tbl:1} and \ref{tbl:2}, respectively. We refer to the abstraction-based method by ``abst'', and our approach based on interval analysis by ``intvl''. Denote by $N_q$ and $N_{\text{trans}}$ the number of abstract states and transitions, respectively. The ratio $W/\Omega$ indicates the coverage of $\Omega$ by the computed winning sets (from which the system can stay in $\Omega$) in terms of volume.

\begin{table}[htbp]
\small \caption{Comparison of two methods in case $\Omega_1$}
\label{tbl:1}
\centering
\begin{tabular}{c||c|c|c|c}
\hline
 & $N_q$ & $N_{\text{trans}}$ & $W/\Omega$ & Time(s)\\
\hline
abst ($\eta=0.001$) & 1881 & 145415 & $2.8\%$ & 41.089\\
\hline
abst ($\eta=0.004$) & 125 & 2912 & $0\%$ & -- \\
\hline
intvl ($\varepsilon=0.001$) & 156 & 2132 & $99.2\%$ & 1.886\\
\hline
\end{tabular}
\end{table}

\begin{table}[htbp]
\small \caption{Comparison of two methods in case $\Omega_2$}
\label{tbl:2}
\centering
\begin{tabular}{c||c|c|c|c}
\hline
 & $N_q$ & $N_{\text{trans}}$ & $W/\Omega$ & Time(s) \\
\hline
abst ($\eta=0.001$) & 1311 & 0 & $0\%$ & --\\
\hline
abst ($\eta=0.0005$) & 5421 & 4990 & $0\%$ & --\\
\hline
intvl ($\varepsilon=0.001$) & 187 & 2259 & $99.1\%$ & 2.312 \\
\hline
\end{tabular}
\end{table}

With the specification $\Omega_1$, the pendulum is to be controlled around the upright position. From Table \ref{tbl:1}, to obtain an nonempty winning set, one has to use a grid width as small as $0.001$ when applying an abstraction-based method. This will generate a large number of discrete states and transitions. With the same precision, our method yields less states and transitions but a considerably bigger winning set, which almost covers the entire set $\Omega$. The advantage of our method is even more clear for the specification $\Omega_2$, where the system has to be stabilized around a positive angle. Abstraction-based methods fail in control synthesis because the growth bound used gives a conservative estimation of state trajectories.

The simulation results shown in Figure \ref{fig:ipdlctl} indicate that the invariance specifications are achieved in both situations.

\begin{figure}[htbp]
  \centering
  \includegraphics[scale=0.65]{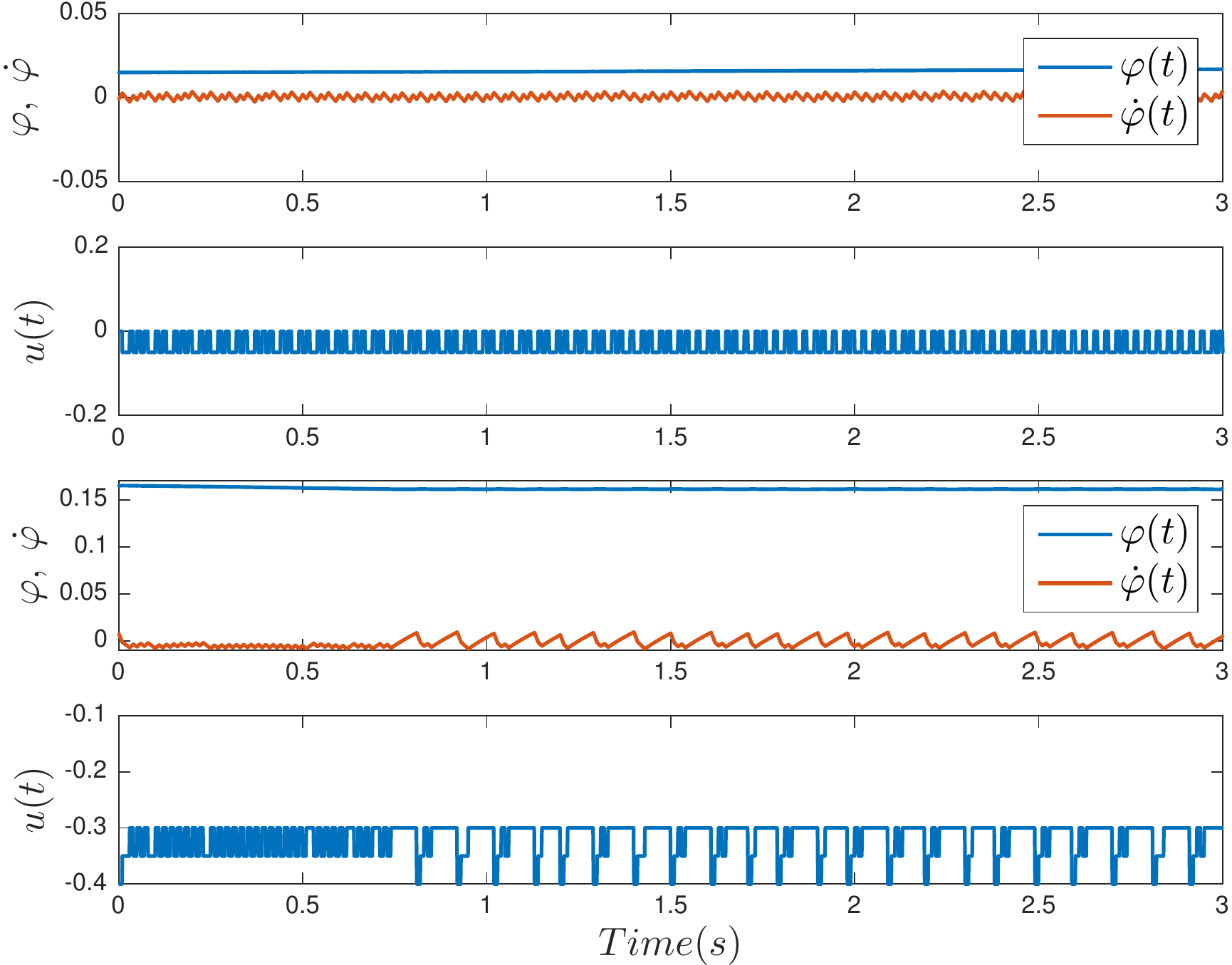}
  \caption{Closed-loop responses and the switching signals. The curves of $\varphi(t),\dot{\varphi}(t)$ and $u(t)$ on the top are the results for the first case $\Omega_1, x_0=[0.015, -0.001]$; the curves on the bottom are for the second case $\Omega_2, x_0=[0.165, 0.008]$. }
  \label{fig:ipdlctl}
\end{figure}

\subsection{Polynomial dynamical system}
Consider a nonlinear switched system with four modes of polynomial dynamics \cite{LiuOTM13}:
\begin{equation*}
  \begin{aligned}
    f_1(x)&=
    \begin{bmatrix}
      0.85x_1-0.1x_2-0.05x_1^3 \\
      x_2-0.1x_2^2+0.1x_1+0.2
    \end{bmatrix},\\[2\jot]
    f_2(x)&=
    \begin{bmatrix}
      0.85x_1-0.1x_2-0.05x_1^3 \\
      0.9x_2+0.1x_1
    \end{bmatrix},\\[2\jot]
    f_3(x)&=
    \begin{bmatrix}
      0.99x_1-0.02x_2-0.01x_1^3+0.04 \\
      x_2+0.02x_1+0.2
    \end{bmatrix},\\[2\jot]
    f_4(x)&=
    \begin{bmatrix}
      0.99x_1-0.02x_2-0.01x_1^3-0.03 \\
      x_2+0.02x_1-0.2
    \end{bmatrix}.
  \end{aligned}
\end{equation*}

The invariance control target set is given by $\Omega=[0.2, 3]\times[-2, -0.5]$, which only contains an unstable fixed point of mode 1 $[0.8952, -1.7015]$. Displayed in Figure \ref{fig:poly4} (a) by gray intervals, the inner approximation of the maximal controlled invariant set is obtained using Algorithm \ref{alg:innerc}. It shows that $\Omega$ is not controlled invariant itself. We apply the same control policy as in the previous two cases to design an invariance controller. The simulation result with initial condition $x_0=[0.208,-1.06]$, shown in Figure \ref{fig:poly4}, indicates that the system is controlled inside the target region as required.

\begin{figure}[htbp]
  \centering
  \begin{subfigure}[b]{0.5\textwidth}
    \centering
    \includegraphics[scale=0.65]{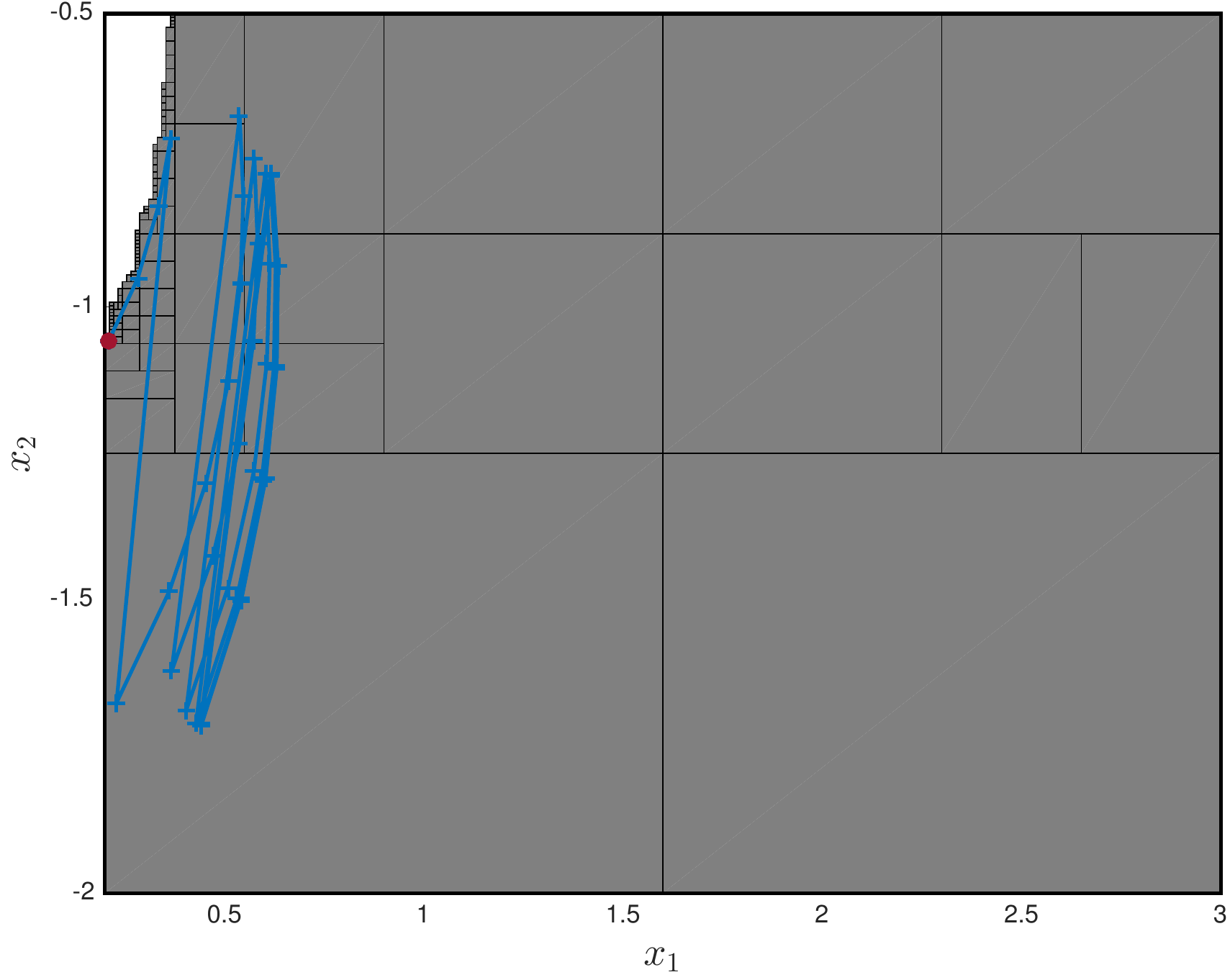}
    \caption{Closed-loop phase portrait.}
  \end{subfigure}
  \begin{subfigure}[b]{0.5\textwidth}
    \centering
    \includegraphics[scale=0.65]{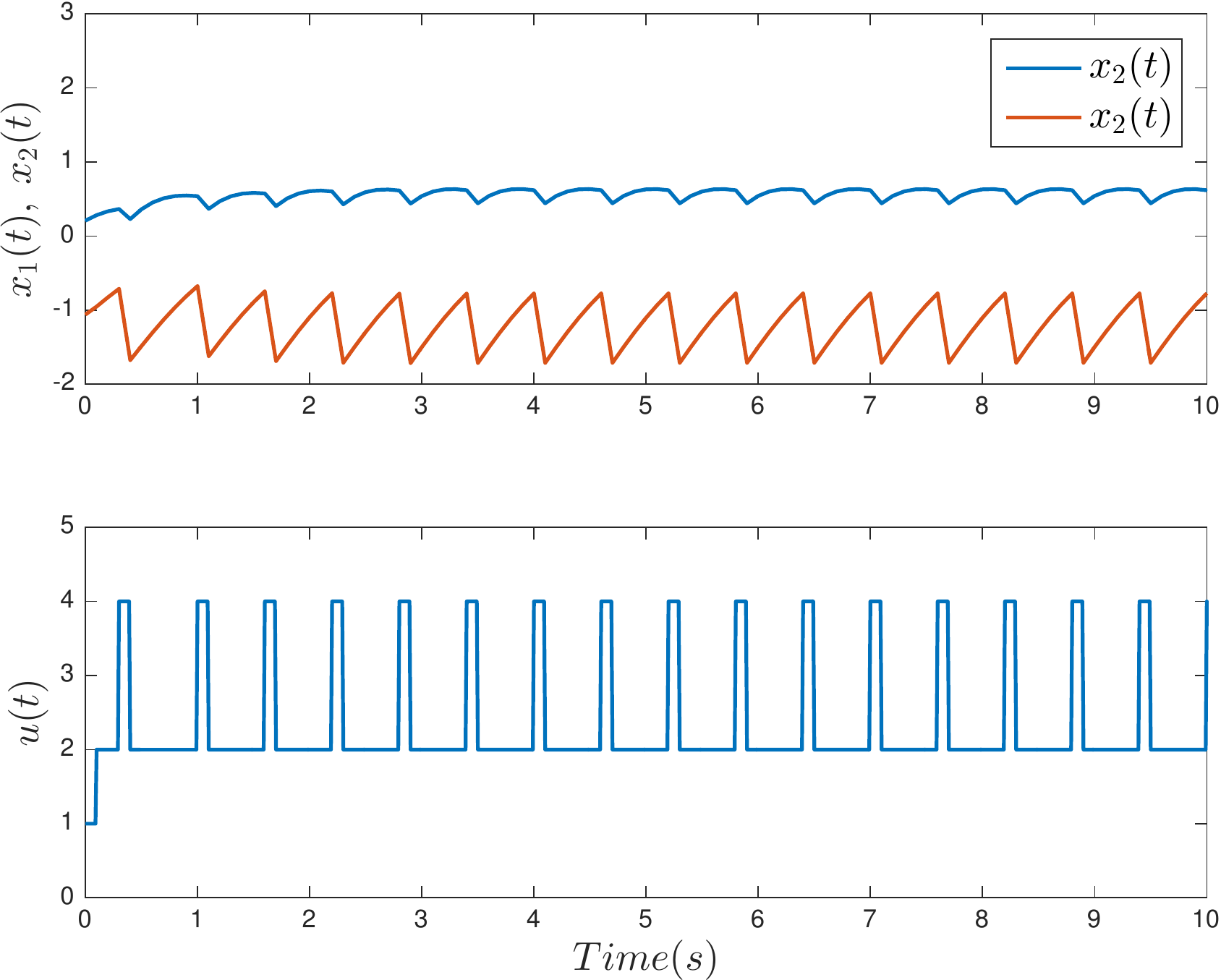}
    \caption{Time history of closed-loop states and control variables.}
  \end{subfigure}
  \caption{Invariance control results for switched polynomial system.}
  \label{fig:poly4}
\end{figure}

\section{Conclusion}
We have presented an interval analysis approach to the invariance control problem for discrete-time switched nonlinear systems. The approach does not assume that subsystems have asymptotically stable dynamics nor that they have common equilibrium points. The critical step is the computation of maximal controlled invariant sets. Using interval methods, we developed algorithms for computing outer and inner approximations of such invariant sets by unions of intervals. We proved that the maximal controlled invariant sets can be outer approximated in arbitrary precision. As a main result of this paper, we introduced a robustly controlled invariance condition for the finite determination of invariant inner approximations. This condition also implies the existence of a partition-based invariance controller, which can be extracted from invariant inner approximations of the maximal controlled invariant sets. Experimental studies showed that this approach is effective and efficient for switched systems of a moderate scale.
  
We demonstrated that interval methods are promising and efficient for solving low-dimensional formal control synthesis problems. To make interval methods applicable to higher-dimensional systems, our future research will focus on two aspects: 1) design efficient data structures that are suitable for parallel computing; 2) investigate model reduction techniques in the context of formal control synthesis. Extending this work on invariance control, future work will consider more general specifications and continuous-time dynamical systems.

\section*{Acknowledgment}
The research was supported in part by the Natural Sciences and Engineering Council of Canada.

\ifCLASSOPTIONcaptionsoff
  \newpage
\fi

\bibliographystyle{IEEEtranS}
\bibliography{../invctlref}

\end{document}